\documentclass{amsart}

\usepackage[margin=3cm, right=4cm]{geometry}
\usepackage{amssymb, amsmath, amsthm}
\usepackage{mathabx}
\usepackage{graphicx}
\usepackage[utf8]{inputenc}
\usepackage{hyperref}
\usepackage{young}
\usepackage{epstopdf}
\usepackage{mathtools}
\usepackage{stmaryrd}
\usepackage{rotating}
\usepackage{tabularx}
\usepackage[mark=o]{dynkin-diagrams}
\usepackage{subcaption} 
\usepackage{longtable}
\usepackage{rotating}
\usepackage[all]{xy}
\usepackage{dsfont}
\usepackage{enumitem}
\usepackage{cleveref}
\usepackage{multirow}
\usepackage{float}
\usepackage{array}
\usepackage{booktabs}
\usepackage[linesnumbered, vlined, boxed, ruled]{algorithm2e}
\usepackage{algorithmic}

\newtheorem{thm}{Theorem}[section]
\newtheorem{cor}[thm]{Corollary}
\newtheorem{lem}[thm]{Lemma}
\newtheorem{prop}[thm]{Proposition}

\newtheorem{rem}[thm]{Remark}
\newtheorem{ex}{Example}
\numberwithin{equation}{section}

\newcommand{\C}{\mathbb{C}}
\newcommand{\R}{\mathbb{R}}
\newcommand{\Z}{\mathbb{Z}}
\newcommand{\N}{\mathbb{N}}

\newcommand{\weyl}{\mathcal{W}}
\newcommand{\schub}{\mathcal{S}}

\renewcommand{\mod}{\ \mathrm{mod}\ }
\newcommand{\g}{\mathfrak{g}}
\newcommand{\Aab}{A_{\alpha+\beta}}

\newcommand{\mab}{m_{\alpha,\beta}}
\newcommand{\nab}{n_{\alpha,\beta}}

\newcommand{\height}{\mathrm{ht}}
\newcommand{\rr}[1]{{\mathbf{#1}}}
\newcommand{\rrh}[1]{{\widetilde{\mathbf{#1}}}}
\newcommand{\rrt}[2]{\widehat{\mathbf{#1}}_{#2}}
\newcommand{\upla}[1]{\mathbf{#1}}

\DeclareMathOperator{\ad}{ad}
\DeclareMathOperator{\Ad}{Ad}

\allowdisplaybreaks

\title[Integral Homology and Poincaré Polynomials of Real Flag Manifolds]{Integral Homology and Poincaré Polynomials of classical and exceptional Real Flag Manifolds}

\author{Jordan Lambert Silva}
\address{UFF}
\email{jordanlambert@id.uff.br}
\author{Lonardo Rabelo}
\address{UFJF}
\email{lonardo.rabelo@ufjf.br}

\thanks{This work was supported by the FAPERJ (Carlos Chagas Filho Foundation for Supporting Research in the State of
Rio de Janeiro) no. 010.002602/2019, and FAPEMIG (Foundation for Supporting Research in the State of
Minas Gerais) RED-00133-21. Research developed with the assistance of CENAPAD-SP (National Center for High Performance Computing in São Paulo), project UNICAMP/FINEP--MCTI.
}



\begin{document}

\begin{abstract}
This paper computes the integral homology of real flag manifolds associated with split real forms of classical and exceptional semisimple Lie algebras. Using the cellular homology provided by the Bruhat decomposition, we introduce a unified framework to systematically determine the coefficients of the boundary operator, explicitly resolving the issue of calculating their signs. This is achieved by computing the degree of change of coordinate maps between different reduced decompositions of Weyl group elements, analyzing commutation and braid relations through Lie bracket computations and exponential identities. By adopting the normal form of Weyl group elements as a canonical choice for reduced decompositions, we establish an explicit algorithmic implementation for these homology computations. As a direct application, we derive the Poincaré polynomials for the classical types $B_n, C_n$, and $D_n$ for $n \leqslant 7$, and for the exceptional types $F_4, E_6$, and $E_7$. With the aid of these polynomials, we address the question of the orientability of split real flag manifolds of exceptional Lie algebras.
\end{abstract}

\maketitle

\section{Introduction}

This article is devoted to the computation of the integral homology of real flag manifolds associated with split real forms of both classical and exceptional Lie algebras. We follow the approach developed by the second author and SanMartin \cite{RSm19}, which uses the Bruhat decomposition together with explicit parameterizations of Schubert cells to derive a formula for the coefficients of the boundary operator in cellular homology. Several aspects of this topic are worth highlighting.

The first aspect is the degree of difficulty that real flag manifolds present compared to the complex ones. The latter have cells only in even dimensions, so their homology is generated by Schubert cells, whereas the former have cells in all dimensions, and the boundary operator is generally non-trivial. 

Another important aspect concerns the different approaches to working with real flag manifolds and the choice of (co)homology theory. The pioneering work of Kocherlakota \cite{Koc95} introduces an algorithm to compute coefficients based on the Morse-Witten complex. However, it remains incomplete, as it does not address the determination of the signs of the coefficients. Subsequently, Casian-Stanton \cite{CS99} obtained the differentials using the machinery of infinite-dimensional representation theory. These works represent substantial progress and play a key role in the development of the subject. Nonetheless, a careful examination of the literature reveals a lack of explicit accounts of how these results can be effectively implemented in concrete calculations. This absence may be attributed to the inherent challenges posed by these approaches, particularly when dealing with the complexities involved in obtaining the signs. 

As observed by Matszangosz \cite{Mat19}, the determination of signs plays a crucial role in effective calculations. This issue was theoretically addressed in \cite{RSm19} using a cellular homology approach based on a more detailed analysis of the Bruhat order in Weyl groups. In \cite{Mat19}, an alternative method for computing signs is given using the cohomology of Vasiliev's complex. Through explicit calculations for flag manifolds of type $A$, this work provides, for the first time, an algorithm that effectively addresses the issue of signs, at least within this context. More recently, Hudson-Matszangosz-Wendt \cite{HMW23} have shown that all torsion in integral cohomology is $2$-torsion by computing the Witt-sheaf cohomology rings of partial flag manifolds of type $A$. From a general perspective, in rational cohomology, it is worth mentioning the work of He \cite{He2020}, which describes the cohomology rings of flags of type $A$ in terms of characteristic classes and exterior algebras related to the cohomology rings of real Stiefel manifolds. 

Independently of the approach, a practical challenge is translating theoretically obtained formulas into effective computations. So far, the authors' work has relied on developing combinatorial tools to describe the Weyl group of each specific type. The most notable case to date has been that of split real flag manifolds of type $A$, for which a formula for the coefficients of the cellular boundary operator was obtained using certain diagrams associated with the permutation code (also called Lehmer code) of $S_n$ (see \cite{LR26}).

The breakthrough in this work lies in offering a unified approach that applies uniformly to any (split) real flag manifold. The computation of the boundary operator starts with a fixed choice of reduced decompositions for the elements of the Weyl group parameterizing the Schubert cells. The coefficients are expressed in terms of covering pairs, that is, pairs $w, w'$ such that $w \leq w'$ in the Bruhat order and $\ell(w) = \ell(w') + 1$. However, the decomposition fixed for $w'$ does not necessarily agree with the one induced by the choice made for $w$. Consequently, the coefficient formula involves an additional factor accounting for this discrepancy, which is directly responsible for the sign (see Theorem \ref{thm:split}). A key contribution of this paper is to provide an explicit and systematic method to compute this factor, given by the degree of a map (the change of coordinate map) relating two reduced decompositions of the same Schubert cell. We determine these maps by decomposing the problem according to the type of operation in the Weyl group: commutation, short, and long braid relations (see Section \ref{sec:changeofcoordinate}). This generalizes the method developed for type $A$ (\cite{LR26}, Section 4.3) and required computations involving Lie brackets and exponential identities (see Section \ref{sec:liebrackets}). As a result, we obtain explicit formulas for the degree factor that determines the sign (see Propositions \ref{prop:degree1} and \ref{prop:degree}).

Another contribution of this work is the use of the normal form of Weyl group elements as a canonical choice of reduced decompositions.
 This choice not only enables an efficient computational implementation of the homology groups, but also provides a precise description of the elementary moves (commutations and braid relations) and their effect on the degree (see Subsection \ref{subsec:normalform}). We present an explicit algorithm based on this framework for carrying out the computations (see Subsection \ref{subsec:computational}).

As an application of these methods, and with the aid of computational techniques, we derive the Poincaré polynomials for real flag manifolds associated with split real forms of semisimple Lie groups. We provide formulas for the classical types \(B_n, C_n,\) and \( D_n\), for $n\leqslant 7$, as well as explicit formulas for the exceptional types \(F_4, E_6,\) and \(E_7\). Notably, in several cases, these formulas contain expressions of type $A$. To the best of our knowledge, this is the first time such formulas have appeared in the literature. Due to their importance and relevance, we have chosen to present these results already in the next section. Finally, once the polynomials have been computed, we were able to validate the results by performing a simple dimension count to determine which real flag manifolds of split real forms of exceptional Lie algebras are orientable.

The complete source code and supporting results are publicly available on GitHub and archived on Zenodo \cite{lambert_zenodo}.

The paper is organized as follows. We present the formulas for the Poincaré polynomials in Section \ref{sec:formulas}. Section \ref{sec:preliminaries} introduces the basics of the Bruhat decomposition and the cellular structure given by the Schubert cells. Section \ref{sec:liebrackets} is devoted to deriving Lie bracket formulas and exponential identities. In Section \ref{sec:changeofcoordinate}, we describe the change of coordinate map for the three types of moves: commutation, short, and long braid.
Section \ref{sec:homology} explains the details regarding the computation of the cellular homology, introduces the normal forms, and provides an algorithmic framework for determining boundary maps. Section \ref{sec:orientability} concerns the orientability of split flags of some exceptional Lie algebras.

\subsection{Flag Manifolds}

Flag manifolds are defined as homogeneous spaces $G/P_\Theta$, where $G$ is a non-compact semisimple Lie group and $P_\Theta$ is a parabolic subgroup of $G$.
The flag manifolds for various groups $G$ with non-compact real semisimple Lie algebra $\mathfrak{g}$ are identical.
For further details in notation and the basics on flag manifolds used here, see \cite{RSm19, LR22a}.

If $\mathfrak{g} = \mathfrak{k} \oplus \mathfrak{s}$ is a Cartan decomposition, let $\mathfrak{a}$ be a maximal abelian sub-algebra contained in $\mathfrak{s}$. A sub-algebra $\mathfrak{h} \subset \mathfrak{g}$ is said to be a Cartan sub-algebra if $\mathfrak{h}_{\mathbb{C}}$ is a Cartan sub-algebra of $\mathfrak{g}_{\mathbb{C}}$. Let $\Pi$ be the set of roots of the pair $(\mathfrak{g}, \mathfrak{a})$ and  fix a simple system of roots $\Sigma \subset \Pi$. If $\mathfrak{h} = \mathfrak{a}$ is a Cartan sub-algebra of $\mathfrak{g}_\C$, we say that $\mathfrak{g}$ is a split real form of $\mathfrak{g}_{\mathbb{C}}$. A property of split real forms is that all restricted root spaces $\mathfrak{g}_\alpha$ are one-dimensional (i.e., the root multiplicities are $d_\alpha = 1$ for all $\alpha \in \Pi$). Consequently, for each simple root $\alpha \in \Sigma$, the rank-one subalgebra $\mathfrak{g}(\alpha)$ generated by $\mathfrak{g}_{\pm \alpha}$ is isomorphic to $\mathfrak{sl}(2,\mathbb{R})$.

Denote by $\Pi^{\pm}$, respectively, the set of positive and negative roots, and by $\mathfrak{a}^+$ the Weyl chamber $\mathfrak{a}^+ = \{ H \in \mathfrak{a} \colon \alpha (H) > 0 \mbox{ for all }  \alpha \in \Sigma \}$. The direct sum of root spaces corresponding to the positive roots is denoted by $\mathfrak{n} = \sum_{\alpha \in \Pi^+} \mathfrak{g}_\alpha$. The Iwasawa decomposition of $\mathfrak{g}$ is given by $\mathfrak{g} = \mathfrak{k} \oplus \mathfrak{a} \oplus \mathfrak{n}$. The notations $K$ and $N$ refer, respectively, to the connected subgroups whose Lie algebras are $\mathfrak{k}$ and $\mathfrak{n}$.

 A minimal parabolic sub-algebra of $\mathfrak{g}$ is given by $\mathfrak{p} = \mathfrak{m} \oplus \mathfrak{a} \oplus \mathfrak{n}$, where $\mathfrak{m}$ is the centralizer of $\mathfrak{a}$ in $\mathfrak{k}$. Let $P$ be the
 minimal parabolic subgroup with Lie algebra $\mathfrak{p}$. Note that $P$ is the normalizer of $\mathfrak{p}$ in $G$.
 We call $\mathbb{F} = G/P$ the maximal flag manifold of $G$ and denote by $b_0$ the base point $1 \cdot P$ in $G/P$.

 Now, assume that $\Theta\subset \Sigma$ is any subset of simple roots. Such a choice provides a very interesting way to obtain several flag manifolds, called partial flag manifolds, as we now explain. Denote by $\mathfrak{g}(\Theta)$ the semisimple Lie algebra generated by $\mathfrak{g}_{\pm \alpha}$, $\alpha \in \Theta$. Let $G(\Theta)$ be the connected group with Lie algebra $\mathfrak{g}(\Theta)$. Let $\mathfrak{n}_\Theta$ be the sub-algebra generated by the root spaces $\mathfrak{g}_{-\alpha}$, $\alpha \in \Theta$, and consider $\mathfrak{p}_\Theta = \mathfrak{n}_\Theta \oplus \mathfrak{p}$. The normalizer $P_\Theta$ of $\mathfrak{p}_\Theta$ in $G$ is a standard parabolic subgroup that contains $P$. Finally, the corresponding flag manifold $\mathbb{F}_\Theta = G/P_{\Theta}$ is called a partial flag manifold of $G$ of type $\Theta$. We denote by $b_\Theta$ the base point $1\cdot P_\Theta$ in $G/P_\Theta$.  Given a choice of order for the set $\Sigma=\{\alpha_1,\ldots, \alpha_n\}$, denote by $\mathbf{k} = \mathbf{k}(\Theta) = (k_1, \ldots, k_r)$ the indices of the complement of $\Theta$; i.e., $\Sigma \setminus \Theta = \{\alpha_{k_1}, \ldots, \alpha_{k_r}\}$, where $k_1 < k_2 < \cdots < k_r$.

 A partial flag manifold is the base space for the natural equivariant fibration $\pi _{\Theta }:\mathbb{F}\rightarrow \mathbb{F}_{\Theta }$, whose fiber is $P_{\Theta }/P$. This fiber is a flag manifold of a semisimple Lie group $M_{\Theta }\subset G$, whose rank is the order of $\Theta $. 

\section{Main computational results}\label{sec:formulas}

Let $G$ be a non-compact semisimple Lie group whose Lie algebra is a split real form. For a given subset of simple roots $\Theta$, we consider the partial flag manifold $\mathbb{F}_\Theta$.

Given $n\in\N$, the $q$-analogue of $n$ and $n!$ are, respectively, the polynomials $(n)_t = 1 + t + \cdots + t^{n-1}=\frac{1-t^n}{1-t}$ and $(n)_t!=(1)_t(2)_t\cdots (n)_t = (1+t)(1+t+t^2)\cdots (1+t+ \cdots + t^{n-1})$.
Given $n_1,\dots,n_k$ non-negative integers such that $n=n_1+\cdots+n_k$, define the $q$-multinomial coefficient by
\begin{equation*}
    \binom{n}{n_1,\dots,n_k}_{t} = \frac{(n)_t!}{(n_1)_t!\cdots (n_k)_t!}.
\end{equation*}

Given $n$ positive and $\mathbf{k} = (0<k_1<\cdots<k_r\leqslant n)$, define the number $X_n(\mathbf{k})$ by
\begin{equation}\label{eq:Lvalue}
X_n(\mathbf{k}) = \left\lfloor \dfrac{k_{1}}{2} \right\rfloor + \left\lfloor \dfrac{k_{2}-k_{1}}{2} \right\rfloor + \cdots + \left\lfloor \dfrac{k_{r}-k_{r-1}}{2} \right\rfloor + \left\lfloor \dfrac{n+1-k_{r}}{2} \right\rfloor,
\end{equation}
where $\lfloor \cdot \rfloor$ is the floor function. Consider  $X_n(\emptyset) = 0$. 

For a flag manifold $\mathbb{F}_{\Theta}$, the integer homology groups are
\begin{equation*}
    H_i(\mathbb{F}_{\Theta}, \Z) \cong \Z^{\beta_i} \oplus (\Z_2)^{T_i}
\end{equation*}
where $\beta_i$ are the Betti numbers and $(\Z_2)^{T_i}$ is the torsion. The Poincaré polynomial in integer coefficients is the generating polynomial of the Betti numbers.

We begin by recalling the Poincaré polynomials of real flag manifolds of type $A_n$, as obtained in \cite{He2020}.

\begin{thm}[\cite{He2020}, Prop. 3.14]
The Poincaré polynomial for partial flag manifolds of the flag manifold $Sl(n+1,\R)/P_\Theta$ is
\begin{equation*}
PA_n^\Theta(t) = q_A(t) \cdot \binom{X_n(\mathbf{k})}{ \left\lfloor \frac{k_{1}}{2} \right\rfloor , \left\lfloor \frac{k_{2}-k_{1}}{2} \right\rfloor , \cdots , \left\lfloor \frac{n+1-k_{r}}{2} \right\rfloor}_{t^{4}}\cdot \prod_{i = X_n(\mathbf{k})}^{\left\lfloor \frac{n}{2}\right\rfloor-1} (1+t^{4i+3}),
\end{equation*}
where
\begin{equation*}
    q_A(t)= 
    \begin{cases}
        1+t^{n}, & \mbox{ if $n$ odd and $X_n(\mathbf{k})<\frac{n+1}{2}$}\\
        1, & \mbox{otherwise.}
    \end{cases}
\end{equation*}
\end{thm}

Now, we present formulas for partial flag manifolds of all other types. Assume that the roots of the Weyl group are labeled as Bourbaki:
\begin{align*}
B_n&: \dynkin[labels={1,,,,n}] B{} &
C_n&: \dynkin[labels={1,,,,n}] C{} &
D_n&: \dynkin[labels={1,,,,n-1,n}] D{} \\
F_4&: \dynkin[label, ordering=Bourbaki] F4 &
E_6&: \dynkin[label, ordering=Bourbaki] E6 &
E_7&: \dynkin[label, ordering=Bourbaki] E7 &
E_8&: \dynkin[label, ordering=Bourbaki] E8
\end{align*}

To visually denote a subset of simple roots $\Theta \subset \Sigma$, we modify the full Dynkin diagram of $\Sigma$. The roots that are {removed} from $\Sigma$ to form $\Theta$ are marked with a cross ($\dynkin A{X}$).
The remaining unmarked nodes represent $\Theta$. If removing nodes splits the diagram into disconnected components, $\Theta$ is identified by the product of the Lie algebras for each component.

For example, let $\Sigma$ be the roots of $B_5$. To represent the subset $\Theta = \{\alpha_2, \alpha_4, \alpha_5\}$, we remove $\alpha_1$ and $\alpha_3$ from the diagram:
$\Theta = \dynkin B{XoXoo}$. This yields two remaining components and defines the set as $\Theta = A_1 \times B_2$.

For any set of simple roots $\Sigma$, we say that the Dynkin diagram of a subset $\Theta \subseteq \Sigma$ is a subdiagram of $A_k$ if it can be obtained by deleting nodes from the Dynkin diagram of $A_k$.

The following Theorems \ref{thm:typeB}, \ref{thm:typeC}, and \ref{thm:typeD} were calculated and verified for $n\leqslant 7$, as available in the GitHub repository and archived on Zenodo \cite{lambert_zenodo}.

\begin{thm}[Type $B_n$]\label{thm:typeB}
The Poincaré polynomial for partial flag manifolds $SO(2n+1, \R)/P_\Theta$ of type $B_n$, for $n\leqslant 7$, is
\begin{equation*}
    PB_n^\Theta(t)= q_B(t) \cdot PA_{n-1}^{\Theta-\{\alpha_n\}}(t) \cdot \prod_{i = \left\lfloor \frac{n+1-k_r}{2}\right\rfloor}^{\left\lfloor \frac{n}{2}\right\rfloor-1} (1+t^{4i+3}),
\end{equation*}
where
\begin{equation*}
    q_B(t)= 
        \begin{cases}
            1, & \mbox{ if $n$ is even}\\
            1+t^{n}, & \mbox{ if $n$ is odd and } X_{n-1}(\mathbf{k})<\frac{n+1}{2}\\
            1+t^{2n+1}, & \mbox{ if $n$ is odd and } X_{n-1}(\mathbf{k})=\frac{n+1}{2}.
        \end{cases}
\end{equation*}
\end{thm}

\begin{thm}[Type $C_n$]\label{thm:typeC}
The Poincaré polynomial for partial flag manifolds $Sp(2n, \R)/P_\Theta$ of type $C_n$, for $n\leqslant 7$, is
\begin{equation*}
    PC_n^\Theta(t)= PA_{2\left\lfloor \frac{n}{2}\right\rfloor}^{\Theta-\{\alpha_n\}}(t) \cdot \prod_{i = \left\lfloor \frac{n+1-k_r}{2}\right\rfloor}^{\left\lfloor \frac{n+1}{2}\right\rfloor-1} (1+t^{4i+1}).
\end{equation*}
\end{thm}

\begin{thm}[Type $D_n$]\label{thm:typeD}
Suppose that $\alpha_n\not\in \Theta$. In this case, the Poincaré polynomial for partial flag manifolds $SO(2n, \R)/P_\Theta$ of type $D_n$, for $n\leqslant 7$, is
\begin{equation*}
    PD_n^\Theta(t)= PA_{n-1}^{\Theta}(t) \cdot \prod_{i = 0}^{\left\lfloor \frac{n-1}{2}\right\rfloor-1} (1+t^{4i+3}).
\end{equation*}
\end{thm}

In type $D_n$, if $\alpha_n \in \Theta$ but $\alpha_{n-1} \not\in \Theta$, consider $\Psi$ where we just switch $\alpha_n$ by $\alpha_{n-1}$. We then have the homeomorphism $\mathbb{F}_\Theta \cong \mathbb{F}_\Psi$ and can also apply Theorem \ref{thm:typeD}. Hence, this theorem corresponds to 75\% of all partial flag manifolds of type $D$ where both $\alpha_{n-1}$ and $\alpha_n$ do not belong simultaneously to $\Theta$.

In the analysis of case type D, we observe that there is a ``non-standard behavior'' of the Poincaré polynomial when both roots $s_{n-1}, s_{n}$ belong to $\Theta$. For example, the Poincaré polynomial for the flag manifold given by $D_4$ and $\Theta = \dynkin D{oXoo}$ is $PD_4^\Theta = 3t^7 + 2t^4 + 1$.

\begin{thm}[Type $F_4$]
    The Poincaré polynomial for partial flag manifolds of type $F_4$ is given by Table \ref{tbl:F4}.

\renewcommand{\arraystretch}{1.2}

\begin{table}[ht]
  \centering
  \caption{Poincaré polynomial $PF_4^\Theta$}\label{tbl:F4}
  \begin{tabular}{cc}
    \toprule
    $\Theta$ & $PF_4^\Theta$ \\
    \midrule
    $\emptyset$ & $(t^{11} + 1)(t^7 + 1)(t^3 + 1)^2$\\ \midrule
    $A_1$, $A_2$ & $(t^{11} + 1)(t^7 + 1)(t^3 + 1)$\\ \midrule
    $A_1\times A_1$, $A_2\times A_1$, $B_2$ & $(t^{11} + 1)(t^7 + 1)$ \\ \midrule
    $B_3$ & $t^{11} + 1$ \\ \midrule
    $C_3$ & $t^{15} + 1$ \\
    \bottomrule
  \end{tabular}
\end{table}

\end{thm}

\begin{thm}[Type $E_6$]   
The Poincaré polynomial for partial flag manifolds of type $E_6$ is given by Table \ref{tbl:E6}.

\begin{table}[ht]
  \centering
  \caption{Poincaré polynomial $PE_6^\Theta$}\label{tbl:E6}
  \begin{tabular}{cc}
    \toprule
    $\Theta$ & $PE_6^\Theta$ \\
    \midrule
    $D_4$  &  $(t^{8} + 1)(t^{16} + t^{8} + 1)=\binom{3}{1,1,1}_{t^8}$\\ \midrule
    $D_5$ & $t^{16} + t^{8} + 1=\binom{3}{1,2}_{t^8}$ \\
    \midrule
    subdiagram of $A_8$ & $PA_8^{\Theta}(t)$\\
    \bottomrule
  \end{tabular}
\end{table}
\end{thm}

\begin{thm}[Type $E_7$]
    The Poincaré polynomial for partial flag manifolds of type $E_7$ can be divided into two cases:

    \begin{itemize}
        \item $\Theta$ contains the three roots $\{\alpha_2,\alpha_5,\alpha_7\}$: There are 15 partial flag manifolds. Their Poincaré polynomials are given in Table \ref{tbl:E7_1}. Denote $h(t) = (t^{13}+ 1)(t^{27} + t^{25} + t^{23} + t^{21} + t^{19} + t^{17} + t^{16} + t^{12} + 2 t^8 + t^4 + 1)$.

        \begin{table}[ht]
  \centering
  \caption{Poincaré polynomial $PE_7^\Theta$ where $\Theta$ contains the three roots $\{\alpha_2,\alpha_5,\alpha_7\}$}\label{tbl:E7_1}
  \begin{tabular}{cc}
    \toprule
    $\Theta$ & $PE_7^\Theta$ \\ \midrule
     \dynkin E{ooXoooo} & $(t^{15} + 1)(t^{13} + 1)(t^9 + 1)\binom{3}{1,2}_{t^4}$\\ \midrule
    \dynkin E{oooXooo}, \dynkin E{ooXooXo}  & \multirow{2}{*}{${(t^{15} + 1)(t^{13} + 1)(t^9 + 1)\binom{3}{1,2}_{t^4}}\binom{3}{1,2}_{t^4}$}\\ 
    \dynkin E{ooXXooo}, \dynkin E{XooXooo} & \\  \midrule
\dynkin E{oooXoXo}, \dynkin E{ooXXoXo}  & \multirow{2}{*}{${(t^{15}+1)(t^{13} + 1)(t^9 + 1)\binom{3}{1,2}_{t^{4}}}\binom{3}{1,1,1}_{t^{4}}$}\\ 
   \dynkin E{XooXoXo}  & \\  \midrule     
     \dynkin E{XoXoooo} & $(t^{15}+ 1)(t^{13}+ 1)(t^{11}+ 1)(t^{9}+ 1)$\\ \midrule
     \dynkin E{XoXooXo}, \dynkin E{XoXXooo} & ${(t^{15}+1)(t^{13}+ 1)(t^{11}+1)(t^{9}+ 1)}\binom{3}{1,2}_{t^{4}}$\\ \midrule
     \dynkin E{XoXXoXo} & ${(t^{15}+1)(t^{13}+ 1)(t^{11}+1)(t^{9}+ 1)}\binom{3}{1,1,1}_{t^{4}}$\\ \midrule
     \dynkin E{Xoooooo} & ${t^{29} + t^{27} + t^{25} + t^{21} + t^{16} + t^{12} + t^8 + 1}$\\ \midrule
     \dynkin E{oooooXo} & $h(t)$ \\ \midrule
\dynkin E{XooooXo} & $h(t)\binom{2}{1,1}_{t^{8}}$\\ 
    \bottomrule
  \end{tabular}
\end{table}

        \item $\Theta$ does not contain the three roots $\{\alpha_2,\alpha_5,\alpha_7\}$ simultaneously: The Poincaré polynomial is given in Table \ref{tbl:E7_2}.

\begin{table}[ht]
  \centering
  \caption{Poincaré polynomial $PE_7^\Theta$ where $\Theta$ does not contain the three roots $\{\alpha_2,\alpha_5,\alpha_7\}$ simultaneously} \label{tbl:E7_2}
  \begin{tabular}{cc}
    \toprule
    $\Theta$ & $PE_7^\Theta$ \\
    \midrule
    $D_4$  &  $(t^{13}+1)(t^9+1)(t^5 +1)\binom{3}{1,1,1}_{t^8}$\\ \midrule
    $D_5$ & $(t^{13}+1)(t^9+1)(t^5 +1)\binom{3}{1,2}_{t^8}$ \\
    \midrule
    $E_6$ & $(t^{13}+1)(t^9+1)(t^5 +1)$ \\
    \midrule
    subdiagram of $A_9$ & $(t^{13}+1)(t^5 +1) \cdot PA_{9}^{\Theta}(t)$\\
    \bottomrule
  \end{tabular}
\end{table}
\end{itemize}
\end{thm}

\section{Cellular decomposition of flag manifolds}\label{sec:preliminaries}

To justify the homology computations underlying the Poincaré polynomials presented in Section \ref{sec:formulas}, we must review the topological structure of real flag manifolds. This structure is provided by the Bruhat decomposition, which partitions the manifold into Bruhat cells. The closures of these cells, the Schubert cells, naturally endow the flag manifold with a cellular complex structure.

Before proceeding, we fix some standard notation. Let $\N=\{1,2,3, \dots\}$ and $\Z$ be the set of integers. For $n,m\in\Z$, where $n\leqslant m$, denote the set $[n,m]=\{n,n+1, \dots, m\}$. For $n\in \N$, denote $[n]=[1,n]$.

\subsection{Weyl Group}

The Weyl group $\weyl$ of the root system $\Pi$ is the group generated by the reflections $s_{\alpha}$, $\alpha \in \Pi$. For the subset $\Theta \subset \Sigma$, we define the subgroup $\mathcal{W}_\Theta$ generated by the reflections with respect to the roots $\alpha \in \Theta$. 
We denote by $\mathcal{W}^{\Theta}$ the subset of minimal representatives of the cosets of $\mathcal{W}_{\Theta}$ in $\mathcal{W}$.

If we consider the elements of $\mathcal{W}$ as a product of simple reflections $s_i=s_{\alpha_i}$, $\alpha_i \in \Sigma$, the length $\ell(w)$ of $w \in \mathcal{W}$ is defined as the number of simple reflections in any reduced decomposition of $w$. 

The Bruhat-Chevalley order (or simply strong Bruhat order) is defined as it follows: we say that $w_{1}\leqslant w_{2}$ if given a reduced decomposition $w_2 = s_{j_1} \cdots s_{j_r}$ then $w_{1}=s_{j_{i_{1}}}\cdots s_{j_{i_{k}}}$ for some $1\leqslant i_1\leqslant \cdots \leqslant i_k\leqslant r$. 

When there exists $w,w' \in \weyl$ such that $w'\leqslant w$ and $\ell(w) = \ell(w')+1$, we say that $w$ covers $w'$ (alternatively, $w,w'$ is a covering pair). If $w$ covers $w'$ and $w=s_{i_1}\cdots s_{i_\ell}$ is a reduced decomposition, then we will denote by $I$ the integer in $[\ell]$ such that $w' = s_{i_1} \cdots \widehat{s_{i_I}}\cdots s_{i_\ell}$, where the integer $I$ depends on $w'$ and the choice of the reduced decomposition of $w$. 

Let $M_{\ast}$ and $M$ denote, respectively, the normalizer and the centralizer of $\mathfrak{a}$ in $K$. Then the Weyl group $\mathcal{W}$  can equivalently be identified with the quotient $M_{\ast}/M$. Denote by $\mathfrak{m}$ the Lie algebra of $M$.

\subsection{Cellular decomposition}

The Bruhat decomposition presents flag manifolds as a union of $N$-orbits, namely,
\begin{equation*}
\mathbb{F}_{\Theta }=\coprod_{w\in \mathcal{W}^{\Theta }} N\cdot wb_{\Theta },
\end{equation*}
where $b_{\Theta} = 1\cdot P_{\Theta}$ is the base point. Each orbit $N\cdot wb_{\Theta}$, for $w \in \mathcal{W}$, is called a Bruhat cell. It is diffeomorphic to a Euclidean space and, in the case of a split real form, its dimension coincides with the length of $w$, i.e., $\dim \left( N\cdot wb_{\Theta }\right) = \ell(w)$ (for details, see  Wiggerman \cite{Wig98}, Cor. 2.6). A Schubert cell $\mathcal{S}_{w}^{\Theta}$ is the closure of a Bruhat cell. The Bruhat-Chevalley order also characterizes a partial order between the corresponding Schubert cells. It equips the flag manifolds with a cellular structure where $\mathcal{S}_{w}^{\Theta}= {\bigcup N\cdot ub_{\Theta}}$, with the union over $u\in\weyl^\Theta$ such that $u\leqslant w$.

\subsection{Gluing cells}

In order to understand how the Schubert cells glue to each other, we start with the maximal flag manifold.
For this part, we refer to \cite{RSm19}.
A Schubert cell $\mathcal{S}_{w}$ is obtained from smaller cells $\mathcal{S}_{v}$, for $v<w$, by gluing a cell of dimension $\dim (N\cdot wb_{0})$. Once this process is performed for each $w\in \mathcal{W}$, we get a cellular decomposition for $\mathbb{F}$ that is explicitly given by characteristic maps. In order to define a characteristic map for $\mathcal{S}_{w}$, $w\in 
\mathcal{W}$, we must choose a reduced expression 
\begin{equation*}
w=s_{1}\cdots s_{n}
\end{equation*}
as a product of simple reflections $s_{i}=s_{\alpha _{i}}$.

As before, $\mathbb{F}=G/P$ is the maximal flag manifold. We denote by $\mathbb{F}_{i}=G/P_{i}$ the partial flag manifolds where $P_{i}=P_{\{\alpha_{i}\}}$, for $\alpha_{i}$ a simple root. The canonical fibration is $\pi
_{i}:\mathbb{F}\rightarrow \mathbb{F}_{i}$. Denote $K_{i}=K\cap P_{i}$. 
\begin{prop}[\cite{RSm19}, Prop. 1.3]
    The Schubert cell for a reduced decomposition $w=s_{1}\cdots s_{n}$ is given by
    \begin{equation*}
    \mathcal{S}_{w} = K_{1} \cdots K_{n} \cdot b_{0}.
\end{equation*}

\end{prop}

The strategy proceeds by parameterizing each $K_i$ that appears in the preceding product.

\begin{lem}[\cite{RSm19}, Lem. 1.6]\label{parametrization}
Let $G=G(\alpha )$ be a real rank one Lie group with maximal compact subgroup $K=K_{\alpha}$ and the corresponding flag
manifold $\mathbb{F}=S^{m}$ with origin $b_{0}$. Let $B^{m}$ be the closed
ball in $\mathbb{R}^{m}$. Then, there exists a continuous map $\psi
:B^{m}\rightarrow K$ such that

\begin{itemize}
\item $\psi (S^{m-1}) \subset M$ and hence $\psi(S^{m-1}) \cdot b_0 = b_0$.

\item If $x\in B^{m}\setminus S^{m-1}$, then $\psi (x)\cdot wb_{0}$ is a
diffeomorphism in the Bruhat cell that is the complement of $b_0$.
\end{itemize}
\end{lem}

Let $\mathbb{F}$ be a real flag manifold of a Lie algebra $\mathfrak{g}$ that is a split real form. For each root $\alpha\in\Pi$, we have $\g(\alpha) \cong \mathfrak{sl}(2,\R)$, and the compact Lie algebra of $K_{\alpha}$ is $\mathfrak{so}(2)$. Given a Cartan involution $\theta$, let $X_\alpha\in\g_\alpha$ and $X_{-\alpha} = -\theta(X_\alpha)\in\g_{-\alpha}$ be such that $\left\langle X_\alpha,X_{-\alpha}\right\rangle = \frac{2}{\langle\alpha,\alpha\rangle}$. Then $[X_\alpha,X_{-\alpha}]=H_\alpha^\vee = \frac{2 H_\alpha}{\langle\alpha,\alpha\rangle}$. Denote $A_\alpha = X_\alpha - X_{-\alpha}$. Explicitly, we can write $\rho: \mathfrak{sl}(2,\R)\to \g(\alpha)$ with $\rho(H)=H_\alpha^\vee$, $\rho(X)=X_\alpha$, and $\rho(Y)=X_{-\alpha}$ as follows:
\begin{align*}
    H &= \begin{pmatrix}
        1 & 0\\
        0 & -1
    \end{pmatrix} ,
    & X &= \begin{pmatrix}
        0 & 1\\
        0 & 0
    \end{pmatrix},
    & Y &= \begin{pmatrix}
        0 & 0\\
        1 & 0
    \end{pmatrix}.
\end{align*}

\begin{lem}[\cite{RSm19}, Lem. 1.7]
\label{compact_parametrization_1} The one-dimensional version of Lemma \ref{parametrization} is realized by 
\begin{eqnarray*}
\psi :[0,\pi ]\rightarrow K_{\alpha }\,,\,t\mapsto \exp (tA_{\alpha }).
\end{eqnarray*}%
In particular, $\psi (0)=1$ and $\psi (\pi )=\exp (\pi A_{\alpha})$.
\end{lem}

Given a simple root $\alpha$, under the identification $\weyl \cong M_*/M$, the reflection corresponding to $\alpha
$ is represented by $s_{\alpha} = \psi(\frac{\pi}{2}) = \exp(\frac{\pi}{2}A_\alpha)$.

It follows that, for each $i$, there exists $\psi _{i}:B^{d_{i}}\rightarrow K_{i}$, where $d_{i}$ is the dimension of the fiber of $\mathbb{F}\rightarrow \mathbb{F}_{i} $, that is, the dimension of the flag of $G(\alpha _{i})$. Let $B_{w}=B^{d_{1}}\times \cdots \times B^{d_{n}}$ be the ball of dimension $d=d_{1}+\cdots +d_{n}$. Then the characteristic map $\Phi
_{w}:B_{w}\rightarrow \mathbb{F}$ is defined by 
\begin{equation*}
\Phi _{w}(t_{1},\ldots ,t_{n})=\psi _{1}(t_{1})\cdots \psi _{n}(t_{n})\cdot
b_{0}.
\end{equation*}

\begin{rem} Different decompositions of $w$ lead to distinct characteristic maps. Therefore, the notation $\Phi_w$ should implicitly refer to a reduced decomposition of $w$, for instance $\Phi_{s_1 \cdots s_n}$. To preserve this simplified notation, we will henceforth fix a set of reduced expressions for each $w \in \mathcal{W}$.
\end{rem}

\begin{prop}[\cite{RSm19}, Prop. 1.9]
\label{characteristic_map}Let $w=s_{1}\cdots s_{n}$ be a reduced decomposition. Let $\Phi _{w}:B_{w}\rightarrow \mathbb{F}$ be the map defined above and take $\mathbf{t}=(t_{1},\ldots ,t_{n})\in B_{w}$. Then, $\Phi _{w}$ is a characteristic map for $\mathcal{S}_{w}$, that is,

\begin{enumerate}
\item $\Phi_w (B_w) \subset \mathcal{S}_w$.

\item $\Phi _{w}(\mathbf{t})\in \mathcal{S}_{w}\setminus N\cdot wb_{0}$ if
and only if $\mathbf{t}\in \partial B_{w}=S^{d-1}$.

\item $\Phi |_{B_{w}^{\circ }}:B_{w}^{\circ }\rightarrow N\cdot wb_{0}$ is a
diffeomorphism ($B_{w}^{\circ }$ is the interior of $B_{w}$).
\end{enumerate}
\end{prop}

Let $d=\dim \mathcal{S}_{w}=\dim N\cdot wb_{0}$. The
sphere $S^{d}$ is the quotient $B_{w}/\partial (B_{w})$ where the boundary
is collapsed to a point. We can do the same with the Schubert cell $\mathcal{S}_{w}$. 
Define $\sigma _{w}=S_{w}/(S_{w}\setminus N\cdot wb_{0})$, i.e.,
the space obtained by identifying the complement of the Bruhat cell $
\mathcal{S}_{w}\setminus N\cdot wb_{0}$ in $\mathcal{S}_{w}$ to a point. As $
\Phi _{w}(\partial (B_{w}))\subset S_{w}\setminus N\cdot wb_{0}$, it follows
that $\Phi _{w}$ induces a map $S^{d}\rightarrow \sigma _{w}$ which is a
homeomorphism. The inverse of this homeomorphism will be denoted by 
\begin{align}\label{Sphere}
\Phi _{w}^{-1}\colon\sigma _{w}\rightarrow S^{d} 
\end{align}
(although this is not the same as the inverse of $\Phi _{w}$).

The following proposition provides a characterization of the covering pairs, which will prove to be very useful.

\begin{prop}[\cite{RSm19}, Prop. 1.10]
\label{propappendix} Let $w,w^{\prime }\in \mathcal{W}$. The following
statements are equivalent.

\begin{enumerate}
\item $\mathcal{S}_{w^{\prime }}\subset \mathcal{S}_{w}$ and $\dim 
\mathcal{S}_{w}-\dim \mathcal{S}_{w^{\prime }}=1$.

\item If $w=s_{1}\cdots s_{n}$ is a reduced expression of $w\in \mathcal{
W}$ as a product of simple reflections, then

\subitem (i) $w^{\prime }=s_{1}\cdots \widehat{s_{i}}\cdots s_{n}$ is a reduced
expression.

\subitem (ii) If $s_{i}=s_{\alpha _{i}}$ then $\mathfrak{g}(\alpha
_{i})\cong \mathfrak{s}\mathfrak{l}(2,\mathbb{R})$. This is the same as
saying that the fiber of $\mathbb{F}\rightarrow \mathbb{F}_{i}$ has dimension $1$.
\end{enumerate}
\end{prop}

\begin{rem} Given $w^{\prime }$ as above, the decomposition $w^{\prime
}=s_{1}\cdots \widehat{s_{i}}\cdots s_{n}$ is unique. In fact, if $w=s_{1}\cdots
s_{i}\cdots s_{j}\cdots s_{n}$ and $w^{\prime }=s_{1}\cdots s_{i}\cdots \widehat{s_{j}}\cdots s_{n}$, then $s_{i+1}\cdots s_{j}=s_{i}\cdots s_{j-1}$, which
cannot happen.
\end{rem}

\section{Lie bracket formulas and exponential identities}\label{sec:liebrackets}

In this section, we establish the explicit Lie theoretic formulas required to get the signs in the boundary map of cellular homology. Indeed, these formulas will help to deal with different reduced decompositions for a given element of the Weyl group, as we will see in the next section.

This approach relies on computing the adjoint action of the Weyl group representatives $\exp(\pi A_\alpha)$ on the root spaces $\mathfrak{g}_\beta$. To evaluate this action explicitly, the rotation $\exp(\pi A_\alpha)$ must be diagonalized, which requires complex coefficients. However, a fundamental technical obstruction arises over the reals since the rank-one subgroup generated by $\mathfrak{g}(\alpha) \cong \mathfrak{sl}(2, \mathbb{R})$ is not simply connected. To address this issue, we perform these exponentiations in the complexification $\mathfrak{g}_{\mathbb{C}}$ and its associated simply connected group $SL(2, \mathbb{C})$. This ensures the existence of a unique global representation $\Phi$ satisfying $e^{\ad\circ\phi(X)} = \Phi(\exp(X))$, which justifies the operator identity $e^{\ad(\pi A_\alpha)} = e^{\ad(i \pi H_\alpha^\vee)}$ (\cite{RSm19}, Eq. (5)). Applying this diagonal operator to the complexified root space $\mathfrak{g}_{\mathbb{C}, \beta}$ yields the scalar eigenvalue $e^{i \pi \langle \beta, \alpha^\vee \rangle} = (-1)^{\langle \beta, \alpha^\vee \rangle}$. Since the real root space $\mathfrak{g}_\beta$ is an invariant subspace under the action of $M$, this complex lifting isolates the exact real parity that provides the attaching maps of the Schubert cells, allowing us to proceed with the fundamental bracket relations (see \cite{Wig98}, Lem 4.2, \cite{RSm19}, Lem. 1.8).

Let $\mathfrak{g}_\C$ be the complexified semisimple Lie algebra and $\mathfrak{h}$ be a Cartan sub-algebra with root system $\Pi$ and a set of simple roots $\Sigma$. Consider $X_\alpha$ as defined previously. For the pair $\alpha,\beta\in \Pi$, define
\begin{equation}\label{eq:defmab}
[X_{\alpha},X_{\beta}]=m_{\alpha,\beta}X_{\alpha+\beta},
\end{equation}
with $\mab=0$ if and only if $\alpha+\beta$ is not a root. It satisfies $m_{-\alpha,-\beta}=-\mab$ and $m_{\beta,\alpha}=-\mab$. Moreover, $m_{\alpha,\beta}$ is real (see \cite{Kna02}, Thm. 6.6).

Recall that, for each $\alpha \in \Pi$, $A_\alpha=X_\alpha-X_{-\alpha}$. In particular, $A_{-\alpha} = -A_\alpha$ and  
\begin{equation}\label{eq:colcheteA}
    [A_{\alpha},A_{\beta}]=\mab \Aab + m_{-\alpha,\beta}A_{\alpha-\beta}.
\end{equation}

\begin{prop}\label{prop:sequence} Let $\alpha,\beta \in \Sigma$. Then,
    \begin{align*}
        [X_{-\alpha},[X_\alpha,X_\beta]] &= -\frac{2 \langle\beta,\alpha\rangle}{\langle\alpha,\alpha\rangle} X_\beta,\\
        [X_\alpha,[X_{-\alpha},X_{-\beta}]] &= -\frac{2 \langle\beta,\alpha\rangle}{\langle\alpha,\alpha\rangle} X_{-\beta}.
    \end{align*}
\end{prop}
\begin{proof} 
Let $\alpha,\beta\in\Pi$. Consider $p$ and $q$ given by the $\alpha$ string containing $\beta$:
\begin{equation*}
\beta-p\alpha, \ldots, \beta, \ldots, \beta + q\alpha.
\end{equation*}
Since  $\left\langle X_\alpha,X_{-\alpha}\right\rangle = \frac{2}{\langle\alpha,\alpha\rangle}$, it follows from \cite[Cor. 2.37]{Kna02} that
\begin{equation*}
[X_{-\alpha},[X_{\alpha},X_\beta]] = q(p+1)X_{\beta}.
\end{equation*}

When $\alpha$ and $\beta$ are simple roots, $p=0$ and $-q = \frac{2 \langle\beta,\alpha\rangle}{\langle\alpha,\alpha\rangle}$. Similarly, for the negative roots $-\alpha$ and $-\beta$, we have $p=0$ and $-q = \frac{2 \langle-\beta,-\alpha\rangle}{\langle-\alpha,-\alpha\rangle} = \frac{2 \langle\beta,\alpha\rangle}{\langle\alpha,\alpha\rangle}$.
\end{proof}

\begin{cor} Let $\alpha,\beta \in \Sigma$. If the connection between $\alpha$ and $\beta$ is simple, then 
\begin{equation}\label{eq:simpleconnection}
[X_{-\alpha},[X_\alpha,X_\beta]]=[X_\alpha,[X_{-\alpha},X_{-\beta}]]=X_{\beta}.
\end{equation}

If the connection between $\alpha$ and $\beta$ is double, and the Killing number is $\frac{2 \langle\beta,\alpha\rangle}{\langle\alpha,\alpha\rangle}=-2$, i.e., $\alpha$ is the short root, then 
\begin{align}
[X_{-\alpha},[X_\alpha,X_\beta]]&=[X_\alpha,[X_{-\alpha},X_{-\beta}]]=2X_{\beta},\label{eq:doubleconnection1}\\
[X_{-\beta},[X_\beta,X_\alpha]]&=[X_\beta,[X_{-\beta},X_{-\alpha}]]=X_{\alpha}.\label{eq:doubleconnection2}
\end{align}
\end{cor}

Our objective now is to calculate the bracket between the elements $A_{\alpha}$, as given by Equation \eqref{eq:colcheteA}. This involves determining the coefficients $\mab$ for certain choices of $\alpha$ and $\beta$. 

\subsection{Simple edge}
Assume that $\alpha$ and $\beta$ are simple roots connected by a simple edge in the Dynkin diagram. Consequently, $\alpha+\beta$ is a root (implying $m_{\alpha,\beta} \neq 0$), whereas $\alpha-\beta$ is not (implying $m_{-\alpha,\beta}=0$), so that $[A_{\alpha},A_{\beta}]=\mab \Aab$. 

Notice that, by Equation \eqref{eq:colcheteA}, $[A_{\alpha},A_{\alpha+\beta}]=m_{\alpha,\alpha+\beta} A_{2\alpha+\beta}-m_{-\alpha,\alpha+\beta}A_{\beta}$. 
Since $2\alpha+\beta$ is not a root, then $m_{\alpha,\alpha+\beta}=0$. The coefficient $m_{-\alpha,\alpha+\beta}$, according to Equations \eqref{eq:defmab} and \eqref{eq:simpleconnection}, is given by
\begin{align*}
    m_{-\alpha,\alpha+\beta}X_{\beta} = [X_{-\alpha},X_{\alpha+\beta}] = \dfrac{1}{\mab}[X_{-\alpha},[X_{\alpha},X_{\beta}]] = \dfrac{1}{\mab}X_{\beta}.
\end{align*}

We conclude that
\begin{align*}
[A_{\alpha},A_{\alpha+\beta}]&=-\dfrac{1}{\mab}A_{\beta}, \\
[A_{\beta},A_{\alpha+\beta}]&=-\dfrac{1}{m_{\beta,\alpha}}A_{\alpha}=\dfrac{1}{\mab}A_{\alpha}.
\end{align*}

Writing $\ad A_{\alpha},\ad A_{\beta}$ in the basis $\{A_{\alpha},A_{\beta},A_{\alpha+\beta}\}$, we have
\begin{align*}
\ad A_{\alpha} &= 
\begin{pmatrix} 0&0&0\\ 0&0&-\frac{1}{\mab}\\ 0&\mab&0\end{pmatrix}, &
\ad A_{\beta} &= 
\begin{pmatrix} 0&0&\frac{1}{\mab}\\ 0&0&0\\ -\mab&0&0\end{pmatrix}.
\end{align*}

For every $\delta\in\Pi$, we have that $\Ad(\exp (tA_{\delta}))=e^{t\ad A_{\delta}}$. The exponential of the adjoint representation can be obtained as a rotation matrix from Rodrigues' formula in $\R^2$.

\begin{prop}\label{prop:rodriguesformula1}
Let $A$ be the matrix
\begin{equation*}
    A = \begin{pmatrix}
        0 & - \frac{1}{k} \\
        k & 0
    \end{pmatrix}
\end{equation*}
for $k\in\C-\{0\}$. Then, the exponential matrix is
\begin{equation*}
    e^{tA} = (\cos t) I + (\sin t) A =
    \begin{pmatrix}
        \cos t & - \frac{1}{k}\sin t\\
        k\sin t & \cos t
    \end{pmatrix}.
\end{equation*}
\end{prop}
\begin{proof}
This follows easily from the fact that $A^2 = -I$.
\end{proof}

Geometrically, for real values of $k$, the exponential in Proposition \ref{prop:rodriguesformula1} corresponds to an elliptical rotation defined by the equation $k^2 x^2 + y^2 = \lambda$, where $\lambda>0$.

Using Proposition \ref{prop:rodriguesformula1}, we have
\begin{align}
\mathrm{Ad}(\exp tA_{\alpha}) &= 
\begin{pmatrix} 1&0&0\\  0&\cos(t)&-\frac{1}{\mab}\sin(t)\\  0&\mab \sin(t)&\cos(t)\end{pmatrix}, \label{eq:expAalpha} \\
\mathrm{Ad}(\exp tA_{\beta}) &= 
\begin{pmatrix} \cos(t)&0&\frac{1}{\mab}\sin(t)\\  0&1&0\\  -\mab\sin(t)&0&\cos(t)\end{pmatrix}. \label{eq:expAbeta}
\end{align}

Given a simple root $\alpha$, recall that $s_{\alpha} = \exp(\frac{\pi}{2}A_\alpha)$. Then, $s_{\alpha}^{-1} = \exp(-\frac{\pi}{2}A_\alpha)$.

\begin{prop}\label{prop:ligacaosimples} Suppose that $\alpha, \beta \in \Sigma$ are connected by a simple edge. For $t\in[0,\pi]$, it holds that
\begin{align}
    \exp(t A_\alpha) \cdot s_\beta \cdot s_\alpha &= s_\beta \cdot s_\alpha \cdot \exp(t A_\beta),\label{eq:exp1}\\
    s_\alpha \cdot \exp(tA_\beta) \cdot s_\alpha &= s_\beta \cdot \exp\left((\pi - t)A_\alpha\right) \cdot s_\beta. \label{eq:exp2}
\end{align}
\end{prop}
\begin{proof}
    First of all, we have the following relations:
    \begin{align}
        s_{\beta}^{-1}\exp(tA_\alpha)s_\beta &= \exp\left(t\Ad\left(\exp\left(-\frac{\pi}{2} A_{\beta}\right)\right)A_\alpha\right)=\exp(m_{\alpha,\beta}t\Aab);\label{eq:relation1} \\
        s_{\alpha}\exp(tA_\beta)s_\alpha^{-1} &= \exp\left(t\Ad\left(\exp\left(\frac{\pi}{2} A_{\alpha}\right)\right)A_\beta\right)=\exp(m_{\alpha,\beta}t\Aab).\label{eq:relation2}
    \end{align}

    For \eqref{eq:exp1}, using Equations \eqref{eq:relation1} and \eqref{eq:relation2}, we have
    \begin{align*}
        \exp(t A_\alpha) s_\beta s_\alpha &= s_\beta \exp(m_{\alpha,\beta}t\Aab) s_\alpha = s_\beta s_\alpha \exp(t A_\beta).
    \end{align*}

    For \eqref{eq:exp2}, 
    \begin{align*}
        s_\alpha \exp(tA_\beta) s_\alpha &= \exp(m_{\alpha,\beta}t\Aab) s_\alpha s_\alpha = \exp(-m_{\beta,\alpha}t\Aab) s_\alpha s_\alpha \\
        &= s_\beta \exp(-t A_\alpha) s_\beta^{-1} s_\alpha s_\alpha = s_\beta \exp(-t A_\alpha) s_\alpha \exp(m_{\alpha,\beta}\frac{\pi}{2}A_{\alpha+\beta}) s_\alpha \\
        &= s_\beta \exp(-t A_\alpha) s_\alpha s_\alpha s_\beta = s_\beta \exp((\pi-t) A_\alpha) s_\beta. \qedhere
    \end{align*}
\end{proof}

\subsection{Double edge}
Assume that $\alpha,\beta\in \Sigma$ are connected by a double edge in the Dynkin diagram, with $\alpha$ denoting the short root (i.e., $\langle \beta, \alpha^\vee \rangle = -2$). As in the case of a single edge, we have $[A_{\alpha},A_{\beta}]=\mab \Aab$.

By Equation \eqref{eq:colcheteA}, $[A_{\alpha},A_{\alpha+\beta}]=m_{\alpha,\alpha+\beta} A_{2\alpha+\beta} -m_{-\alpha,\alpha+\beta}A_{\beta}$. An analogous calculation to the simple edge case provides $m_{-\alpha, \alpha+\beta}=\dfrac{2}{\mab}$. However, in this case, $2\alpha+\beta$ is a root and $m_{\alpha,\alpha+\beta}\neq 0$. Therefore, we conclude that
\begin{equation*}
[A_{\alpha},A_{\alpha+\beta}]=m_{\alpha,\alpha+\beta} A_{2\alpha+\beta} -\dfrac{2}{\mab}A_{\beta}.
\end{equation*}

By Equation \eqref{eq:colcheteA}, $[A_{\beta},A_{\alpha+\beta}]=m_{\beta,\alpha+\beta} A_{\alpha+2\beta} -m_{-\beta,\alpha+\beta}A_{\alpha}$. It is enough to get the coefficient $m_{-\beta,\alpha+\beta}$ since $m_{\beta,\alpha+\beta}=0$. According to Equations \eqref{eq:defmab} and \eqref{eq:doubleconnection2}, 
\begin{align*}
    m_{-\beta,\alpha+\beta}X_{\alpha} = [X_{-\beta},X_{\alpha+\beta}] = -\dfrac{1}{\mab}[X_{-\beta},[X_{\beta},X_{\alpha}]] = -\dfrac{1}{\mab}X_{\alpha}.
\end{align*}
Hence, we have that
\begin{equation*}
[A_{\beta},A_{\alpha+\beta}]=\dfrac{1}{\mab}A_{\alpha}.
\end{equation*}

By Equation \eqref{eq:colcheteA}, $ [A_{\alpha},A_{2\alpha+\beta}]=-m_{-\alpha,2\alpha+\beta}A_{\alpha+\beta}$ since $m_{\alpha,2\alpha+\beta}=0$. Observe that Equation \eqref{eq:defmab} provides
\begin{align*}
    m_{-\alpha,2\alpha+\beta}X_{\alpha+\beta} & = [X_{-\alpha},X_{2\alpha+\beta}] = \dfrac{1}{m_{\alpha,\alpha+\beta}}[X_{-\alpha},[X_{\alpha},X_{\alpha+\beta}]]\\
    &=\dfrac{1}{\mab\cdot m_{\alpha,\alpha+\beta}}[X_{-\alpha},[X_{\alpha},[X_{\alpha},X_{\beta}]]].
\end{align*}

By the Jacobi identity,
\begin{align*}
[X_{-\alpha},[X_{\alpha},[X_{\alpha},X_{\beta}]]]&=[[X_{-\alpha},X_{\alpha}],[X_{\alpha},X_{\beta}]]+[X_{\alpha},[X_{-\alpha},[X_{\alpha},X_{\beta}]]] \\
&= -[H_{\alpha}^\vee,[X_{\alpha},X_{\beta}]]+2[X_\alpha,X_{\beta}] = 2\mab X_{\alpha+\beta}
\end{align*}
since
\begin{equation*}
[H_{\alpha}^\vee,[X_{\alpha},X_{\beta}]] = (\alpha+\beta)(H_\alpha)\frac{2 \mab}{\langle \alpha,\alpha \rangle}\cdot X_{\alpha+\beta}=\left(2+ \dfrac{2\langle \beta,\alpha \rangle}{\langle \alpha,\alpha \rangle} \right)\mab\cdot X_{\alpha+\beta}=0. 
\end{equation*}

Therefore, $m_{-\alpha,2\alpha+\beta}=\dfrac{2}{m_{\alpha,\alpha+\beta}}$. We conclude that
\begin{equation*}
[A_{\alpha},A_{2\alpha+\beta}]=-\dfrac{2}{m_{\alpha,\alpha+\beta}}A_{\alpha+\beta}.
\end{equation*}

The bracket $[A_{\beta},A_{2\alpha+\beta}]=0$ since both coefficients $m_{\beta,2\alpha+\beta}$ and $m_{-\beta,2\alpha+\beta}$ are zero. 

In summary:
\begin{align*}
[A_{\alpha},A_{\alpha+\beta}]&=m_{\alpha,\alpha+\beta} A_{2\alpha+\beta} -\dfrac{2}{\mab}A_{\beta}, & [A_{\alpha},A_{2\alpha+\beta}]&=-\dfrac{2}{m_{\alpha,\alpha+\beta}}A_{\alpha+\beta}, \\ [A_{\beta},A_{\alpha+\beta}]&= \dfrac{1}{\mab}A_{\alpha}, &
[A_{\beta},A_{2\alpha+\beta}]&=0.
\end{align*}

Expressing $\ad A_{\alpha},\ad A_{\beta}$ in the basis $\{A_{\alpha},A_{\beta},A_{\alpha+\beta}, A_{2\alpha+\beta}\}$, we have
\begin{align*}
\ad A_{\alpha} &= 
\begin{pmatrix} 0&0&0&0\\  0&0&-\frac{2}{\mab}&0\\  0&\mab&0&-\frac{2}{m_{\alpha,\alpha+\beta}} \\  0&0&m_{\alpha,\alpha+\beta}&0\end{pmatrix},&
\ad A_{\beta} &= 
\begin{pmatrix} 0&0&\frac{1}{\mab}&0\\  0&0&0&0\\  -\mab&0&0&0 \\  0&0&0&0\end{pmatrix}.
\end{align*}  

The following result provides a closed form for the exponential map. It can be viewed as a generalization of the standard Rodrigues' rotation formula in $SO(3)$.

\begin{prop}
\label{prop:rodriguesformula2}
Let $A$ be the matrix
\begin{equation*}
    A = \begin{pmatrix}
        0 & -\frac{2}{k} & 0 \\
        k & 0 & -\frac{2}{p} \\
        0 & p & 0
    \end{pmatrix},
\end{equation*}
for $k,p\in\C-\{0\}$. Then, the exponential matrix is
\begin{equation*}
    e^{tA} = I + \left( \frac{\sin 2t}{2} \right) A + \left(\frac{\sin^2 t}{2}\right)A^2 = \begin{pmatrix} \cos^2(t)&-\frac{1}{k}\sin(2t)& \frac{2}{k p}\sin^2(t)\\  \frac{k}{2}\sin(2t)&\cos(2t)&-\frac{1}{p}\sin(2t) \\ \frac{kp}{2}\sin^2(t)&\frac{p}{2}\sin(2t)&\cos^2(t)\end{pmatrix}.
\end{equation*}
\end{prop}
\begin{proof}
    The characteristic polynomial of $A$ is $p(\lambda) = -\lambda^3 - 4\lambda$, and then $A^3 = -4A$. Consequently, for $n \ge 0$, $A^{2n+1} = (-4)^n A$, and $A^{2n+2} = (-4)^n A^2$. Substituting these into the Taylor series of $e^{tA}$:
    \begin{align*}
e^{tA} &= I + \sum_{n=0}^{\infty} \frac{t^{2n+1}}{(2n+1)!}A^{2n+1} + \sum_{n=0}^{\infty} \frac{t^{2n+2}}{(2n+2)!}A^{2n+2} \\
&= I + \left( \sum_{n=0}^{\infty} \frac{(-1)^n (2t)^{2n+1}}{(2n+1)!} \right) \frac{A}{2} + \left( \sum_{n=0}^{\infty} \frac{(-1)^n (2t)^{2n+2}}{(2n+2)!} \right) \frac{A^2}{4} \\
&= I + \frac{\sin(2t)}{2} A + \frac{1 - \cos(2t)}{4} A^2.
\end{align*}
Using the identity $1 - \cos(2t) = 2\sin^2(t)$, the result follows.
\end{proof}

Geometrically, for real values of $k$ and $p$, the exponential in Proposition \ref{prop:rodriguesformula2} corresponds to an elliptic rotation defined by the intersection of the plane $x+\frac{2}{kp}z = \lambda$ and the cylinder $(x-\frac{2}{kp}z)^2+\frac{4}{k^2}y^2=\mu$, where $\lambda,\mu\in\R$, $\mu>0$.

Denote $n_{\alpha,\beta}=\mab\cdot m_{\alpha,\alpha+\beta}$. It follows that
\begin{align}
\displaystyle \mathrm{Ad}(\exp tA_{\alpha}) &= 
\begin{pmatrix} 1&0&0&0\\ 0&\cos^2(t)&-\frac{1}{\mab}\sin(2t)& \frac{2}{\nab}\sin^2(t)\\ 0& \frac{\mab}{2}\sin(2t)&\cos(2t)&-\frac{1}{m_{\alpha,\alpha+\beta}}\sin(2t) \\  0& \frac{\nab}{2}\sin^2(t)&\frac{m_{\alpha,\alpha+\beta}}{2}\sin(2t)&\cos^2(t)\end{pmatrix}, \label{eq:expAalphadupla} \\
\mathrm{Ad}(\exp tA_{\beta}) &= 
\begin{pmatrix} \cos(t)&0&\frac{1}{\mab}\sin(t)&0\\ 0&1&0&0\\ -\mab\sin(t)&0&\cos(t)&0 \\ 0&0&0&1\end{pmatrix}. \label{eq:expAbetadupla}
\end{align}

\begin{rem} Observe that in the basis $\{A_{\alpha},\mab A_{\alpha+\beta}, A_{\beta}-\frac{\nab}{2}A_{2\alpha+\beta},A_{\beta}+\frac{\nab}{2}A_{2\alpha+\beta}\}$, we have
\begin{equation*}
\ad A_{\alpha}=\begin{pmatrix} 0&0&0&0\\ 0&0&2&0\\ 0&-2&0&0 \\ 0&0&0&0\end{pmatrix}
\Rightarrow
\Ad(\exp(tA_{\alpha}))=\begin{pmatrix}
1&0&0&0\\ 0&\cos(2t)&\sin(2t)&0\\  0&-\sin(2t)&\cos(2t)&0 \\  0&0&0&1\end{pmatrix}.
\end{equation*}
\end{rem}

\begin{prop}\label{prop:ligacaodupla}
Suppose the edge between $\alpha$ and $\beta$ is double, with $\alpha$ being the short root. For $t\in[0,\pi]$, it holds that
\begin{align}
    \exp(t A_\alpha)\cdot s_{\beta}\cdot s_{\alpha} \cdot s_{\beta} &= s_{\beta}\cdot s_{\alpha} \cdot s_{\beta} \cdot \exp(t A_\alpha),\label{eq:expB1}\\
    s_{\alpha} \cdot \exp(t A_\beta) \cdot s_{\alpha} \cdot s_{\beta} &=s_{\beta} \cdot s_{\alpha}  \cdot \exp(t A_\beta) \cdot s_{\alpha},\label{eq:expB2}\\
    s_{\alpha} \cdot s_{\beta} \cdot \exp(t A_\alpha) \cdot s_{\beta} &= s_{\beta} \cdot \exp((\pi-t) A_\alpha) \cdot s_{\beta} \cdot s_{\alpha}, \label{eq:expB3}\\
    s_{\alpha} \cdot s_{\beta} \cdot s_{\alpha} \cdot \exp(t A_\beta) &= \exp(t A_\beta) \cdot s_{\alpha} \cdot s_{\beta} \cdot s_{\alpha}.\label{eq:expB4}
\end{align}
\end{prop}

\begin{proof}
The proof is based on the same arguments as those of Proposition \ref{prop:ligacaosimples}, through appropriate conjugations and computing Equations \eqref{eq:expAalphadupla} and \eqref{eq:expAbetadupla} at specific values of $t$. In particular, we get the following useful equations:
\begin{align}
        s_{\beta}^{-1}\exp(tA_\alpha)s_\beta &=
         \exp(m_{\alpha,\beta}t\Aab)=
         s_{\beta}\exp(-tA_\alpha)s_\beta^{-1};
         \label{eq:relation11} \\
         s_{\alpha}\exp(tA_{\alpha+\beta})s_\alpha^{-1} &= 
         \exp(-t\Aab)= s_{\alpha}^{-1}\exp(tA_{\alpha+\beta})s_\alpha;\label{eq:relation12}\\
         s_{\alpha}\exp(tA_{\beta})s_\alpha^{-1} &= \exp\left(\dfrac{n_{\alpha,\beta}}{2}tA_{2\alpha+\beta}\right)=s_{\alpha}^{-1}\exp(tA_{\beta})s_\alpha; \label{eq:relation13}\\
         s_{\beta}\exp(tA_{2\alpha+\beta})s_\beta^{-1} &=\exp(tA_{2\alpha+\beta}). \label{eq:relation14}
\end{align}

From Equations \eqref{eq:relation11} and \eqref{eq:relation13}, we derive
\begin{align}
s_{\beta}s_{\beta}s_{\alpha}^{-1}&=s_{\alpha}s_{\beta}s_{\beta};\label{eq:relation16}\\
    s_\alpha s_\alpha s_{\beta}&=s_{\beta}s_\alpha s_\alpha.\label{eq:relation15}
\end{align}

For \eqref{eq:expB1}, by Equations \eqref{eq:relation11} and \eqref{eq:relation12}, we have
\begin{align*} 
   \exp(t A_\alpha) s_{\beta} s_{\alpha} s_{\beta} &= s_{\beta} \exp(\mab t\Aab) s_{\alpha} s_{\beta}= s_{\beta} s_{\alpha} \exp(-\mab t\Aab) s_{\beta}= s_{\beta} s_{\alpha} s_{\beta} \exp(t A_\alpha).
\end{align*}

For \eqref{eq:expB2}, by Equations \eqref{eq:relation13},\eqref{eq:relation14}, and \eqref{eq:relation15}, we have
\begin{align*}
    s_{\alpha}\exp(t A_\beta)s_{\alpha}s_{\beta} 
    &=  \exp\left(\frac{\nab}{2} t A_{2\alpha+\beta}\right) s_\alpha s_\alpha s_{\beta} = \exp\left(\frac{\nab}{2} t A_{2\alpha+\beta}\right) s_{\beta} s_\alpha s_\alpha\\
    &= s_{\beta} \exp\left(\frac{\nab}{2} t A_{2\alpha+\beta}\right) s_\alpha s_\alpha= s_{\beta} s_{\alpha} \exp(t A_{\beta}) s_{\alpha}.
\end{align*}

For \eqref{eq:expB3}, by Equations \eqref{eq:relation11},\eqref{eq:relation12},\eqref{eq:relation16} and \eqref{eq:relation15}, we have
\begin{align*}
    s_{\alpha} s_{\beta}\exp(t A_\alpha)s_{\beta} 
    &= s_{\alpha}s_{\beta}s_{\beta}\exp( \mab t \Aab) =  s_{\alpha} s_{\beta}s_{\beta} s_{\alpha}^{-1} \exp(-\mab t\Aab) s_{\alpha}\\
    &= s_{\alpha} s_{\alpha} s_{\beta} s_{\beta} s_{\beta}^{-1} \exp(-t A_\alpha)s_{\beta} s_{\alpha} = s_{\beta}s_{\alpha} s_{\alpha}  \exp(-t A_\alpha) s_{\beta} s_{\alpha}\\
    &= s_{\beta} \exp((\pi-t) A_\alpha) s_{\beta}s_{\alpha}.
\end{align*}

For \eqref{eq:expB4}, by Equations \eqref{eq:relation13} and \eqref{eq:relation14}, we have
\begin{align*}
s_{\alpha}s_{\beta}s_{\alpha}\exp(t A_\beta) &= s_{\alpha}s_{\beta}\exp\left(\frac{\nab}{2} t A_{2\alpha+\beta}\right) s_{\alpha} = s_{\alpha}\exp\left(\frac{\nab}{2}tA_{2\alpha+\beta}\right)s_{\beta}s_{\alpha} \\
&= \exp(t A_\beta) s_{\alpha}s_{\beta} s_{\alpha}.\qedhere
\end{align*}
\end{proof}

\section{Change of coordinate map}\label{sec:changeofcoordinate}

Given $w\in\weyl$, let $\rr{w}$ and $\rrh{w}$ be two reduced decompositions of $w$. As mentioned above, the characteristic maps $\Phi_\rr{w}$ and $\Phi_\rrh{w}$ are (possibly) distinct. We may regard them as different choices of coordinates on the Schubert cell $\schub_{w}$, and thus $\Phi_\rrh{w}^{-1}\circ\Phi_\rr{w}$ will be called the change of coordinate map. Notice that such maps are, \emph{a priori}, defined only the interiors of the balls. After collapsing the boundary to a point, one obtains a homeomorphism between spheres. Thus, the usage of the expression ``change of coordinate'' implicitly assumes that we are restricting to the interior of the balls.  

It is known that any reduced decomposition of $w$ can be obtained from another one by a sequence of relations among its generators $\Sigma=\{s_1, \ldots, s_n\}$ subject to the following relations:
\begin{enumerate}
\item Commutation: $s_is_j=s_js_i$, when $|i-j|\geq 2$;
\item Short braid: $s_is_{i+1}s_i=s_{i+1}s_{i}s_{i+1}$ for simple edges;
\item Long braid: $s_is_{i+1}s_is_{i+1}=s_{i+1}s_{i}s_{i+1}s_i$ for double edges.
\end{enumerate}

Every time a new reduced decomposition is obtained by applying one of these relations, we will say that a movement (move) has been made, whether it is a commutation, a short, or a long braid relation. So, we assume that the reduced decompositions $\rr{w}$ and $\rrh{w}$ of $w$ are connected by a sequence of moves $f_1 f_2\cdots f_p$ as follows
\begin{equation*}
    \rr{w} = \rr{w}_0 \xrightarrow{f_1} \rr{w}_1 \xrightarrow{f_2} \cdots \xrightarrow{f_p} \rr{w}_p=\rrh{w},
\end{equation*}
where $\rr{w}_1,\dots,\rr{w}_{p-1}$ are reduced decompositions of $w$.
Then,
\begin{equation*}
    \Phi_\rrh{w}^{-1}\circ\Phi_\rr{w} = \left(\Phi_{\rr{w}_{p}}^{-1}\circ\Phi_{\rr{w}_{p-1}}\right) \circ \cdots \circ 
    \left(\Phi_{\rr{w}_{1}}^{-1}\circ\Phi_{\rr{w}_0}\right),
\end{equation*}
by which we may study the map $\Phi_\rrh{w}^{-1}\circ\Phi_\rr{w}$ by breaking it into pieces $\Phi_{\rr{w}_{k}}^{-1}\circ\Phi_{\rr{w}_{k-1}}$, $k=1,\dots, p$. 

The effect of the moves over a reduced decomposition may be described as follows. The reduced decomposition of $\rr{w}_{k-1}$ can be factorized as $\rr{w}_{k-1}=\rr{l}_k\cdot \rr{m}_k \cdot \rr{r}_k$, where $\rr{m}_k$ represents the subword that will be changed after applying the move $f_k$, while $\rr{l}_k$ and $\rr{r}_k$ denote the subwords on the left and right, respectively, that will not be changed. Applying the move $f_k$ yields the new decomposition $\rr{w}_{k}=\rr{l}_k\cdot \rr{n}_k \cdot \rr{r}_k$, where $\rr{n}_{k}$ is the new word obtained from $\rr{m}_{k}$. That is,  
\begin{equation}\label{eq:movefk}
  \rr{l}_k\cdot \rr{m}_k \cdot \rr{r}_k =\rr{w}_{k-1}  \xrightarrow{f_k}  \rr{w}_{k}=\rr{l}_k\cdot \rr{n}_k \cdot \rr{r}_k.
\end{equation}

Denote $\ell=\ell(w)$.
To define the characteristic maps, let us denote by $\psi_{\mathbf{u}}(\upla{t})$ the product of exponentials associated with a word $\rr{u}$ and a parameter $\upla{t}\in [0,\pi]^{\ell(\rr{u})}$. The parameterizations $\Phi_{\rr{w}_{k-1}}, \Phi_{\rr{w}_{k}}: [0, \pi]^\ell \to \mathcal{S}_w$ decompose as:
\begin{align*}
\Phi_{\rr{w}_{k-1}}(\upla{x}, \upla{y}, \upla{z})
    &= \psi_{\rr{l}_k}(\upla{x}) \cdot \psi_{\rr{m}_k}(\upla{y}) \cdot \psi_{\rr{r}_k}(\upla{z}) \cdot b_0 \\
\Phi_{\rr{w}_{k}}(\upla{x}', \upla{y}', \upla{z}')
    &= \psi_{\rr{l}_k}(\upla{x}') \cdot \psi_{\rr{n}_k}(\upla{y}') \cdot \psi_{\rr{r}_k}(\upla{z}') \cdot b_0.
\end{align*}

Since the prefix $\rr{l}_k$ and the suffix $\rr{r}_k$ are identical in both decompositions, the change of coordinate map $\Phi^{-1}_{\rr{w}_{k}} \circ \Phi_{\rr{w}_{k-1}}$ acts as the identity on the coordinates $\upla{x}$ and $\upla{z}$. Thus, we have $\upla{x}' = \upla{x}$ and $\upla{z}' = \upla{z}$, and the map restricts to a local change of coordinates:
\begin{equation}
\label{eq:commutativechange2}
\Phi^{-1}_{\rr{w}_{k}} \circ \Phi_{\rr{w}_{k-1}}(\upla{x}, \upla{y}, \upla{z}) = (\upla{x}, \upla{y}', \upla{z}).
\end{equation}

It remains to express $\upla{y}'$ with respect to $\upla{y}$, which depends on the change of the word $\rr{m}_{k}$ to $\rr{n}_k$ \eqref{eq:movefk}. This analysis will be performed for each type of move.

\subsection{Commutation}
Here,  $\rr{m}_k=s_is_j$ and $\rr{n}_k=s_js_i$, for $|i-j|\geq 2$. Then, $[A_{\alpha_i},A_{\alpha_j}]=0$, which gives that 
\begin{equation*}
\exp(sA_{\alpha_i})\exp(tA_{\alpha_j})=\exp(tA_{\alpha_j})\exp(sA_{\alpha_i}),
\end{equation*}
for all $s,t \in \mathbb{R}$. In particular, if we write $\upla{y}=(y_1,y_2)$ then $\upla{y}'=(y_2,y_1)$. Hence, the restriction of $\Phi^{-1}_{\rr{w}_k}\circ \Phi_{\rr{w}_{k-1}}$ to the interior of the cube $[0,\pi]^{\ell}$ is given by 
\begin{equation}\label{eq:commutativechange}
\Phi^{-1}_{\rr{w}_{k}}\circ \Phi_{\rr{w}_{k-1}}(\upla{x},\upla{y},\upla{z})= (\upla{x},y_2,y_1,\upla{z}).
\end{equation}

\subsection{Short braid}
Consider $\rr{m}_k=s_is_{i+1}s_i$ and $\rr{n}_k=s_{i+1}s_{i}s_{i+1}$, for some $i\in[n-1]$. Let us write $\upla{y}=(y_1,y_2,y_3)$ and $\upla{y}'=(y_1',y_2',y_3')$. Since $\Phi^{-1}_{\rr{w}_{k}}\circ \Phi_{\rr{w}_{k-1}}$ is a diffeomorphism, the strategy consists of evaluating all coordinates except one at the boundary value $\pi/2$. By Proposition \ref{prop:ligacaosimples}, we conclude that $y_1'=y_3, y_2'=\pi-y_2, y_3'=y_1$. Hence, $\Phi^{-1}_{\rr{w}_{k}}\circ \Phi_{\rr{w}_{k-1}}$ restricted to the interior of the cube $[0,\pi]^{\ell}$ is given by 
\begin{equation}\label{eq:shortbraid}
\Phi^{-1}_{\rr{w}_{k}}\circ \Phi_{\rr{w}_{k-1}} (\upla{x},\upla{y},\upla{z})= (\upla{x},y_3,\pi-y_2,y_1,\upla{z}).
\end{equation} 

\subsection{Long braid}

In this case, $\rr{m}_k=s_is_{i+1}s_is_{i+1}$ and $\rr{n}_k=s_{i+1}s_{i}s_{i+1}s_{i}$, for some $i\in [n-1]$. Assume that $\alpha_i$ is the short root. Let us write $\upla{y}=(y_1,y_2,y_3,y_4)$ and $\upla{y}'=(y_1',y_2',y_3',y_4')$. As before, evaluating all coordinates except one at the boundary value $\pi/2$, and by Proposition \ref{prop:ligacaodupla}, we conclude that $y_1'=y_4, y_2'=\pi-y_3, y_3'=y_2, y_4'=y_1$. Hence, $\Phi^{-1}_{\rr{w}_{k}}\circ \Phi_{\rr{w}_{k-1}}$ restricted to the interior of the cube $[0,\pi]^{\ell}$ is given by
\begin{equation}\label{eq:longbraid1}
\Phi^{-1}_{\rr{w}_{k}}\circ \Phi_{\rr{w}_{k-1}} (\upla{x},\upla{y},\upla{z}) =  (\upla{x},y_4,\pi-y_3,y_2,y_1,\upla{z}).
\end{equation}
Now, if $\alpha_i$ is the long root, the same methods and Proposition \ref{prop:ligacaodupla} give that 
\begin{equation}\label{eq:longbraid2}
\Phi^{-1}_{\rr{w}_{k}}\circ \Phi_{\rr{w}_{k-1}}(\upla{x},\upla{y},\upla{z})= (\upla{x},y_4,y_3,\pi-y_2,y_1,\upla{z}).
\end{equation}

\section{Cellular Homology}\label{sec:homology}

Let us begin by reviewing some main results about the determination of the cellular homology coefficients following \cite{RSm19}.
We start in the context of the maximal flag manifold $\mathbb{F}$. 

Let $\mathcal{C}$ be the $\mathbb{Z}$-module freely generated by the Schubert cells $\mathcal{S}_{w}$, $w\in \mathcal{W}$. The boundary map $\partial$ defined over $\mathcal{C}$ is given by $\partial \mathcal{S}_{w}=\sum_{w^{\prime}}c(w,w^{\prime })\mathcal{S}_{w^{\prime }}$, where $c(w,w^{\prime })\in \mathbb{Z}$ in such way that non-trivial coefficients may occur when $w$ covers $w'$. Furthermore, the non-trivial coefficients must be equal to $\pm 2$ (\cite{RSm19}, Thm. 2.2). 

For each $w\in \mathcal{W}$, we choose to fix a reduced decomposition
\begin{equation*}
\rr{w}=s_{1}\cdots s_{\ell}
\end{equation*}
as a product of simple reflections, with $\ell=\ell(w)$. Given $w'$ covered by $w$, there exists a unique index $I$ such that $w'= s_{1}\cdots \widehat{s_{I}}\cdots s_{\ell}$ (cf. Prop. \ref{propappendix}). We will denote this reduced decomposition of $w'$ by $\rrt{w}{I}$, which may be different from the fixed decomposition $\rr{w}'$.

The characteristic maps for $\mathcal{S}_{w'}$ are given by $\Phi_{\rr{w}'}\colon B_{\rr{w}'}\rightarrow \mathcal{S}_{w'}$ and $\Phi_{\rrt{w}{I}}\colon B_{\rrt{w}{I}}\rightarrow \mathcal{S}_{w'}$, where $B_{\rr{w}'}$ and $B_{\rrt{w}{I}}$ are balls of dimension $\ell(w')$. The first map is obtained by the choice of a reduced decomposition for $\rr{w}'$, whereas the latter follows from the deletion operation. In addition, both maps $\Phi_{\rr{w}'}$ and $\Phi_{\rrt{w}{I}}$ are diffeomorphisms when restricted to the interior of the respective balls.

Let $w,w'$ be a covering pair. It will also be useful to denote $v=s_{1} \cdots s_{I-1}$ and $u=s_{I+1}\cdots s_{\ell}$ such that $w'=v\cdot u$. There is a root (not necessarily simple) $\gamma=u^{-1}(\alpha_{I})$ such that $w = w'\cdot s_{\gamma}$.

The following results show how we determine the value of $c(w,w')$.

\begin{thm}[\cite{RSm19}, Thm. 2.8 and \cite{LR22a}, Thm. 5]\label{thm:split} Assume that $\g$ is a split real form. Let $\gamma$ be the root for which $w = w'\cdot s_{\gamma}$. Then 
\begin{equation}\label{eq:coefficient}
    c(w,w^{\prime})=(-1)^{I}\cdot\deg\left(\Phi_{\rr{w}'}^{-1}\circ \Phi_{\rrt{w}{I}} \right)\cdot(1+(-1)^{\height(\gamma^{\vee})}),
\end{equation}
where $\height(\gamma^{\vee})$ is the height of the coroot $\gamma^{\vee}=\frac{2\gamma}{\langle \gamma,\gamma\rangle}$ in the dual root system $\Pi^{*}$.
\end{thm}

Let $\Theta \subset \Sigma$ and consider the partial flag manifold $\mathbb{F}_{\Theta}$. Recall that the Schubert cells  $\mathcal{S}_{w}^{\Theta}$ are the closure of the Bruhat cells $N\cdot wb_{\Theta}$, for $w\in \mathcal{W}^{\Theta }$.
Let $\mathcal{C}^{\Theta}$ be the $\mathbb{Z}$-module freely generated by $\mathcal{S}_{w}^{\Theta}$, for every element $w$ of $\weyl^{\Theta}$. The boundary maps $\partial^{\Theta} :\mathcal{C}^{\Theta} \rightarrow \mathcal{C}^{\Theta}$ are defined by 
\begin{equation*}
\partial^{\Theta} \mathcal{S}_{w}^{\Theta}=\sum_{w'\in \weyl^{\Theta}}c^{\Theta}(w,w')\mathcal{S}_{w'}^{\Theta}
\end{equation*}
for some coefficients $c^{\Theta}(w,w')\in \mathbb{Z}$. 

\begin{thm}[\cite{RSm19}, Thm. 3.4]\label{teo:mainRS2}
The integral homology of the flag manifold $\mathbb{F}_{\Theta}=G/P_{\Theta}$ is isomorphic to the homology of $(\mathcal{C}^{\Theta},\partial^{\Theta})$, where $\partial^{\Theta}$ is obtained by restricting $\partial$ and projecting it onto $\mathcal{C}^{\Theta}$.
\end{thm}
Hence,
\begin{equation*}
c^{\Theta }(w,w')=c(w,w'),
\end{equation*}
where $w,w'\in\weyl^{\Theta}$. That is, it suffices to know the boundary operator on the maximal flag in order to describe it on the partial flag.

As shown in Equation \eqref{eq:coefficient}, the main difficulty in computing the coefficients lies in determining the degree part. This problem has already been specifically solved for type A (cf. \cite{LR26}). We will now present an approach that works effectively for all types independently.

\subsection{Degree computation}\label{subsec:degreecomp}

Suppose that $\rr{w}$ and $\rrh{w}$ are reduced decompositions of $w$. When both reduced decompositions $\rr{w}$ and $\rrh{w}$ are equal, then $\deg\left(\Phi_{\rrh{w}}^{-1}\circ \Phi_{\rr{w}} \right)=1$. More generally, the degree of the composition $\Phi_{\rr{w}}^{-1}\circ \Phi_{\rrh{w}}$ is the Jacobian determinant of the change of coordinate map, i.e.,
\begin{equation*}
    \deg\left(\Phi_{\rrh{w}}^{-1}\circ \Phi_{\rr{w}} \right) = \det\left(J(\Phi_{\rrh{w}}^{-1}\circ \Phi_{\rr{w}}) \right).
\end{equation*}

This degree is well-defined and can be computed for any pair of reduced decompositions of a given $w\in\weyl$.

\begin{prop}\label{prop:degree1}
Suppose that the reduced decompositions $\rr{w}$ and $\rrh{w}$ of $w$ are connected by a sequence of moves $f_1 f_2\cdots f_p$. Then,
\begin{equation*}
    \deg\left(\Phi_{\rrh{w}}^{-1}\circ \Phi_\rr{w} \right) = (-1)^{\#\{k\colon f_k \mbox{ is a commutation}\}} \cdot (-1)^{\#\{k\colon f_k \mbox{ is a long braid}\}}.
\end{equation*}
\end{prop}

\begin{proof}
Write the sequence of moves connecting the decompositions as
\begin{equation*}
    \rr{w} = \rr{w}_0 \xrightarrow{f_1} \rr{w}_1 \xrightarrow{f_2} \cdots \xrightarrow{f_p} \rr{w}_p=\rrh{w}.
\end{equation*}
By the chain rule, the total degree is the product of the degrees of each move:
\begin{equation*}
    \deg(\Phi_\rrh{w}^{-1}\circ\Phi_\rr{w}) = \prod_{k=1}^{p} \deg\left(\Phi_{\rr{w}_{k}}^{-1}\circ\Phi_{\rr{w}_{k-1}}\right).
\end{equation*}
It suffices to compute the degree for each type of move applied. The Jacobian matrix of the transition $\Phi_{\rr{w}_{k}}^{-1}\circ\Phi_{\rr{w}_{k-1}}$ acts non-trivially only on the coordinates corresponding to the reflections involved in the move. Thus, the determinant is simply the determinant of the block corresponding to the move. 

So, let us compute the determinant of the Jacobian matrix for each type of move:
\begin{enumerate}
\item When $f_k$ is a commutation, according to Equation \eqref{eq:commutativechange},  

\begin{equation*}
\deg\left(\Phi^{-1}_{\rr{w}_k}\circ \Phi_{\rr{w}_{k-1}}\right)= \det \begin{pmatrix}
\mathbf{1}& & & \\   & 0& 1& \\    &1&0 & \\    & & &\mathbf{1}
\end{pmatrix}=-1.
\end{equation*}

\item When $f_k$ is a short braid, according to Equation \eqref{eq:shortbraid},
\begin{equation*}
\deg\left(\Phi^{-1}_{\rr{w}_k}\circ \Phi_{\rr{w}_{k-1}}\right)= \det \begin{pmatrix}
\mathbf{1}& & & &\\   & 0& 0& 1& \\    & 0&-1&0 & \\   & 1&0&0& \\    & & & &\mathbf{1}
\end{pmatrix}=1.
\end{equation*}

\item When $f_k$ is a long braid and the first root is short, according to Equation \eqref{eq:longbraid1}, 
\begin{equation*}
\deg\left(\Phi^{-1}_{\rr{w}_k}\circ \Phi_{\rr{w}_{k-1}}\right)= \det \begin{pmatrix}
\mathbf{1}& & & & &\\   &0& 0& 0& 1& \\    & 0&0&1 &0 & \\   & 0&-1&0&0& \\   & 1&0&0&0& \\     & & & & &\mathbf{1}
\end{pmatrix}=-1.
\end{equation*}

The case where the first root is long (Equation \eqref{eq:longbraid2}) yields the same determinant.
\end{enumerate}

Therefore, commutations and long braid moves interfere in the sign, whereas short braid moves do not.
\end{proof}

\begin{rem}
    Proposition \ref{prop:degree1} generalizes Lemmas 3.8 and 3.9 of \cite{LR26}.
\end{rem}

In this proof, we see that the calculation can be carried out by considering each movement individually. The next step is to observe the displacement of a simple and unique reflection after a move. The following lemma will help us to get this result. 

\begin{lem}\label{lem:degree}
    Let $u,w,v\in\weyl$. If $\rr{u}$ and $\rr{v}$ are a reduced decomposition of $u$ and $v$, and $\rr{w}$, $\rrh{w}$ are two reduced decompositions of $w$, then
\begin{equation*}
    \deg\left(\Phi_{\rr{u}\rr{w}\rr v}^{-1}\circ \Phi_{\rr{u}\rrh{w}\rr{v}} \right) = \deg\left(\Phi_{\rr{w}}^{-1}\circ \Phi_{\rrh{w}} \right).
\end{equation*}
\end{lem}
\begin{proof}
    There exists a sequence of moves that transform $\rr{w}$ into $\rrh{w}$:
    $$\rr{w}= \rr{w}_0 \xrightarrow{f_1} \rr{w}_1 \xrightarrow{f_2} \cdots \xrightarrow{f_p} \rr{w}_p = \rrh{w}.$$
    
    This sequence induces a corresponding sequence of moves on the full decompositions, transforming $\rr{u}\rr{w}\rr{v}$ into $\rr{u}\rrh{w}\rr{v}$:
    $$ \rr{u}\rr{w}_0\rr{v} \xrightarrow{f_1'} \rr{u}\rr{w}_1\rr{v} \xrightarrow{f_2'} \cdots \xrightarrow{f_p'} \rr{u}\rr{w}_p\rr{v}. $$
    
    Since the prefix $\rr{u}$ and the suffix $\rr{v}$ remain fixed, each move $f_k'$ acts exclusively on the indices $\ell(u)<i\leqslant\ell(u)+\ell(w)$.
    
    Consider the change of coordinate map $\Phi_{\rr{u}\rr{w}_{k}\rr{v}}^{-1} \circ \Phi_{\rr{u}\rr{w}_{k-1}\rr{v}}$ associated with the $k$-th move. Let $(\upla{x}, \upla{y},\upla{z})$ be the local coordinates, where $\upla{x} \in \mathbb{R}^{\ell(u)}$ corresponds to the prefix $\rr{u}$, $\upla{y} \in \mathbb{R}^{\ell(w)}$ corresponds to the element $\rr{w_k}$, and $\upla{z} \in \mathbb{R}^{\ell(v)}$ corresponds to the suffix $\rr{v}$. This map acts as the identity on $\upla{x}$ and $\upla{z}$ and applies the transformation of the move $f_k$ on $\upla{y}$.
    Thus, the Jacobian matrix has a block-diagonal structure:
    $$ J(\Phi_{\rr{u}\rr{w}_{k}\rr{v}}^{-1} \circ \Phi_{\rr{u}\rr{w}_{k-1}\rr{v}}) = \begin{pmatrix}
    \mathbf{1} &  &  \\
     & J(\Phi_{\rr{w}_{k}}^{-1} \circ \Phi_{\rr{w}_{k-1}}) & \\
     &  & \mathbf{1}\end{pmatrix}, $$
    where $J(\Phi_{\rr{w}_{k}}^{-1} \circ \Phi_{\rr{w}_{k-1}})$ is the Jacobian block associated with the move $f_k$.
    
    Hence, the degree is given by the product of the determinants.
\end{proof}

\begin{prop}\label{prop:degree}
    Let $w\in \weyl$ and suppose that it can be decomposed as $w=s_\alpha\cdot u=u\cdot s_\beta$ for some $u\in\weyl$ with $u<w$, and $\alpha,\beta\in\Sigma$. If $\rr{u}$ is a fixed reduced decomposition of $u$, then
    \begin{equation*}
    \deg\left(\Phi_{s_\alpha\rr{u}}^{-1}\circ \Phi_{\rr{u}s_\beta} \right) = (-1)^{\ell(u)}.
    \end{equation*}
\end{prop}
\begin{proof}
We proceed by induction on $\ell(u)$. The base case $\ell(u)=0$ is trivial (the degree is $1$). Assume the theorem holds for all $u' \in \weyl$ with $\ell(u') < \ell(u)$.

Consider a sequence of moves $f_1 \cdots f_p$ that transforms the reduced decomposition $s_{\alpha} \rr{u}$ into $\rr{u}s_{\beta}$:
\begin{equation*}
    s_\alpha\rr{u} = \rr{w}_0 \xrightarrow{f_1} \rr{w}_1 \xrightarrow{f_2} \cdots \xrightarrow{f_p} \rr{w}_p=\rr{u}s_\beta.
\end{equation*}

Since the first simple reflection $s_\alpha$ in $\rr{w}_0$ must eventually move to the right to become $s_\beta$, there exists a minimal index $i\in [p]$ such that the move $f_i$ involves the first letter of the word $\rr{w}_{i-1}$.

The moves $f_1, \dots, f_{i-1}$ act only on the subword $\rr{u}$, transforming it into some reduced decomposition $\rrh{u}$. Thus, $\rr{w}_{i-1} = s_\alpha\rrh{u}$.
We can decompose $\rrh{u} = \rr{v}_0 \rr{v}_1$ such that the move $f_i$ acts on the prefix $s_\alpha\rr{v}_0$. Specifically, $f_i$ transforms $s_\alpha\rr{v}_0\rr{v}_1$ into $\rr{v}_0 \tilde{s} \rr{v}_1$ (where $\tilde{s} \in \Sigma$).
The length $\ell(\rr{v}_0)$ is determined by the type of move $f_i$:
\begin{itemize}
    \item If $f_i$ is a commutation, $\ell(\rr{v}_0)=1$;
    \item If $f_i$ is a short braid, $\ell(\rr{v}_0)=2$;
    \item If $f_i$ is a long braid, $\ell(\rr{v}_0)=3$.
\end{itemize}

The degree contribution of this specific move is $(-1)^{\ell(\rr{v}_0)}$ (by Proposition \ref{prop:degree1}).

After the move $f_i$, the word becomes $\rr{w}_i = \rr{v}_0 \tilde{s} \rr{v}_1$.
Let $\widetilde{w} = \rr{v}_0^{-1} w$. Note that $\widetilde{w}$ has reduced decompositions $\tilde{s}\rr{v}_1$ and $\rr{v}_1 s_\beta$. Since $\ell(\rr{v}_1) = \ell(u) - \ell(\rr{v}_0) < \ell(u)$, we can apply the induction hypothesis.

Decompose the transition map according to the sequence of moves
\begin{align*}
\deg(\Phi_{s_\alpha \rr{u}}^{-1}\circ \Phi_{\rr{u}s_\beta}) &=
\deg(\Phi_{s_\alpha \rr{u}}^{-1}\circ \Phi_{s_\alpha \tilde{\rr{u}}})\cdot
\deg(\Phi_{s_\alpha \tilde{\rr{u}}}^{-1}\circ \Phi_{\rr{v}_0 \tilde{s} \rr{v}_1})\cdot
\deg(\Phi_{\rr{v}_0 \tilde{s} \rr{v}_1}^{-1}\circ \Phi_{\tilde{\rr{u}} s_\beta})\cdot
\deg(\Phi_{\tilde{\rr{u}}s_\beta}^{-1}\circ \Phi_{\rr{u}s_\beta}).
\end{align*}

By Lemma \ref{lem:degree}, we have
\begin{align*}
\deg(\Phi_{s_\alpha \rr{u}}^{-1}\circ \Phi_{\rr{u}s_\beta}) &=
\deg(\Phi_{\rr{u}}^{-1}\circ \Phi_{ \tilde{\rr{u}}})\cdot
\deg(\Phi_{s_\alpha {\rr{v}_0}}^{-1}\circ \Phi_{\rr{v}_0 \tilde{s}})\cdot
\deg(\Phi_{\tilde{s} \rr{v}_1}^{-1}\circ \Phi_{{\rr{v}_1} s_\beta})\cdot
\deg(\Phi_{\tilde{\rr{u}}}^{-1}\circ \Phi_{\rr{u}}).
\end{align*}

The first and last factors cancel each other out. The second factor represents the specific move $f_i$ on the prefix; whether it is a commutation or a braid relation, Proposition \ref{prop:degree1} implies that its degree is $(-1)^{\ell(\rr{v}_0)}$. Finally, for the third term, we use the induction hypothesis on $\rr{v}_1$, which contributes $(-1)^{\ell(\rr{v}_1)}$. Since $\ell(\rr{u}) = \ell(\rr{v}_0) + \ell(\rr{v}_1)$, the total degree is $(-1)^{\ell(\rr{u})}$.
\end{proof}

In summary, the degree is given by $-1$ to the number of positions moved to the right by the first element $s_\alpha
$ of $w$. Notice that the element $s_\alpha$ changes during this process whenever a braid move occurs. This proposition is the key idea in understanding how the normal form of the Weyl group elements may be used to compute the degrees. This will be explored in the next section. 

\subsection{Normal form}\label{subsec:normalform}

Let $\mathcal{W}$ be a Weyl group and $\Sigma=\{s_1,\ldots, s_n\}$ be the set of simple roots. Once the elements of a Weyl group admit various reduced decompositions, we introduce the normal form, which represents a choice with nice computational properties. 
Here, we will follow the presentation of this subject according to Björner-Brenti \cite{BB05}. For further details, see also du Cloux \cite{duC99} and Casselman \cite{Cas02}.

Let $\mathcal{R}(w)$ be the set of all reduced decompositions of an element $w\in \mathcal{W}$. The normal form of $w\in \mathcal{W}$ is $\min \mathcal{R}(w)$, where the minimum is taken with respect to the lexicographic order if read backwards. In other words, a normal form expression is defined recursively according to the following conditions: (1) the identity corresponds to the empty sequence of generators; (2) if $w$ has normal form $w=s_1\cdots s_\ell$, then $s_\ell$ is the element of smallest index among the elements $s$ of $\Sigma$ such that $ws<w$ and $s_1\ldots s_{\ell-1}$ is the normal form of $ws_{\ell}$. This normal form as mentioned above is also called \textsc{InverseShortLex}. The other version, called \textsc{ShortLex}, is defined so that $s_1$ is the element of the smallest index such that $sw<w$ and $s_2\ldots s_{\ell}$ is the normal form of $s_1w$. We will denote by $NF(w)$ the normal form of $w\in \mathcal{W}$. 

\begin{ex}
Let us calculate the normal form of $w=21543\in S_5$ in one-line notation. We have that $\mathcal{R}(w)$ is given by
\begin{align*}
\{s_3s_4s_3s_1,s_4s_3s_4s_1,s_3s_4s_1s_3,s_3s_1s_4s_3,s_1s_3s_4s_3,s_4s_3s_1s_4,s_4s_1s_3s_4,s_1s_4s_3s_4\}.
\end{align*}
Therefore, $NF(w)=\min \mathcal{R}(w)=s_3s_4s_3s_1$.
\end{ex} 

Given an element $w\in\mathcal{W}$, the set of right descents of $w$ is defined by $D_R(w)=\{s \in \Sigma \colon \ell(ws)<\ell(w)\}$. Let $s_1$ be the first right descent of $w$, that is, the element $s\in \Sigma$ of the smallest index such that $\ell(ws)<\ell(w)$. Now, let $u=ws_1$. Let $s_2$ be the first right descent of $u$, that is, the element $s \in \Sigma$ of the smallest index such that $\ell(us)<\ell(u)$. This process continues until there are no more right descents, providing a decomposition $ws_1s_2\cdots s_\ell=1$, that is, $w=s_\ell\cdots s_2s_1$ which, by construction, is the normal form of $w$. This provides a way to obtain the normal form of an element $w\in \mathcal{W}$ whenever it is possible to determine the first right descent.

\begin{ex}
For the Lie algebra of type A, it follows that $\mathcal{W}=S_n$ is the symmetric group. We can characterize the set of right descents of a permutation by $D_R(w)=\{ i \in [n-1] \colon w_i>w_{i+1}\}$. Let $w=21543\in S_5$ be as above. In this case, $s_i=(i,i+1)$ and the right action permutes the elements in the \textbf{positions} $i$ and $(i+1)$. Then
\begin{align*}
w=\mathbf{21}543 \xrightarrow{s_1} 12\mathbf{54}3 \xrightarrow{s_3} 124\mathbf{53} \xrightarrow{s_4}
12\mathbf{43}5 \xrightarrow{s_3} 12345
\end{align*}

This process provides $ws_1s_3s_4s_3=1$, that is, $w=s_3s_4s_3s_1=NF(w)$.
\end{ex}

According to \cite{Cas02}, the problem of determining the normal form of $w=s_1\cdots s_\ell$ is equivalent to the following problem: given $w=s_1\cdots s_\ell$ in normal form and $s\in \Sigma$, find the normal form of $sw$.

Define the set of left descents of $w$ by $D_L(w)=\{s \in \Sigma \colon \ell(sw)<\ell(w)\}$. In this direction, we have the following theorem.
\begin{thm}[\cite{BB05}, Thm. 3.4.8] \label{thm:nfprod} Let $w\in W$, $NF(w)=s_{1}\cdots s_{\ell}$, and $s\in \Sigma$. Then, we have:\begin{enumerate}
\item If $s\in D_L(w)$, then there exists $1\leq j\leq \ell$ such that $NF(sw)=s_{1}\cdots \widehat{s_{j}} \cdots s_{\ell}$;
\item If $s\notin D_L(w)$, then there exist $0\leq j\leq \ell$ and $r\in\Sigma$ such that $NF(sw)=s_{1}\cdots s_{j}rs_{j+1}\cdots s_{\ell}$.
\end{enumerate}
\end{thm}

\subsection{Algorithm}\label{subsec:algorithm}

Fix the reduced decomposition for $w \in \mathcal{W}$ given by its normal form, i.e., $\rr{w}=NF(w)=s_1\cdots s_\ell$. 
Consider a covering pair $(w,w')$ and denote $\rr{w}'=NF(w')$.
Since $w'\leq w$, by the strong exchange property, there exists a unique $I\in [\ell]$ such that $w'=s_1\ldots \widehat{s_I}\cdots s_\ell$. We denote this reduced decomposition by $\rrt{w}{I}$.
In this section, we present an algorithm to compute $\deg\left({\Phi}_{\rr{w}'}^{-1}\circ \Phi_{\rrt{w}{I}}\right)$.

We can write $\rrt{w}{I}=\rr{w}_L\cdot \rr{w}_R$, where $\rr{w}_L=s_1\cdots s_{I-1}$ and $\rr{w}_R=s_{I+1}\cdots s_\ell$. Note that, in particular, the subexpressions $\rr{w}_L$ and $\rr{w}_R$ are also in their normal forms. 

We need to obtain $\rr{w}'$ by applying commuting or braid relations to move elements from $\rr{w}_L$ to $\rr{w}_R$ (and conversely, if needed). Starting from $s_{I-1}$ and going down to $s_{1}$ (eventually), each element of $\rr{w}_L$ is inserted, one at a time, into the right term, which is initially $\rr{w}_R$. At each step, this right term is replaced by its corresponding normal form until, at some point in the process, the normal form of $w'$  is obtained. At each step, the number of moves required to reach the new normal form is determined by Theorem \ref{thm:nfprod}(2). Finally, the degree is computed using Proposition \ref{prop:degree}. Algorithm \ref{alg:degree_calculation} summarizes the procedure.

\begin{rem} Since $\ell(w)=\ell(w')+1$, we conclude that for each reflection $s_i$ in $\rr{w}_L$, $i\in[I-1]$, we have $s_i \notin D_L(s_{i+1}\cdots s_{I-1}\rr{w}_R)$; otherwise, by Theorem \ref{thm:nfprod}(1), we would have $\ell(w')<\ell(w)-1$. This justifies the use of item (2) of Theorem \ref{thm:nfprod}.
\end{rem}

\begin{rem}
     The algorithm does not need to go through all the elements of $ \rr{w}_L$. It can stop when it reaches $s_p$ where $p \in [I-1]$ is the smallest integer such that $s_p$ in $NF(w)$ is different from the $p$-th element of $NF(w')$.
\end{rem}

\begin{algorithm}
    \SetAlgoLined
    \KwData{$w,w'\in \weyl$ with $w$ covering $w'$}
    \KwResult{$\deg\left({\Phi}_{\rr{w}'}^{-1}\circ \Phi_{\rrt{w}{I}}\right)\in \{-1,1\}$}
    \BlankLine
    $s_1\cdots s_\ell \gets NF(w)$\;
    Let $I$ be the unique index such that $w' =s_1\ldots \widehat{s_I}\cdots s_\ell$\;
    \If{$NF(w')=s_1\ldots \widehat{s_I}\cdots s_\ell$}{    
        \Return $1$\;
    }
    $u=u_1\cdots u_k \gets s_{I+1}\cdots s_\ell$\;
    $\text{degree} \gets 1$\;
    Let $p$ be the smallest index in $[I-1]$ such that the $p$-th element of $NF(w)$ differs from the $p$-th element of $NF(w')$\;
    \For{$s$ from $s_{I-1}$ down to $s_p$}{
        Find $j\in [0,k]$ such that $NF(su) = u_1\cdots u_j r u_{j+1}\cdots u_{k}$\; 
        $\text{degree} \gets \text{degree} \cdot (-1)^{j}$\; 
        $u=u_1\cdots u_{k+1} \gets NF(su)$\;
        $k \gets k+1$\;
    }
    \Return $\text{degree}$\;
\caption{Degree Calculation}\label{alg:degree_calculation}
\end{algorithm}

\begin{ex} Consider the pair $w=436125$ and $w'=432165$ inside $S_6$. Lemma 2.2 of \cite{BB05} implies that $w$ covers $w'$ since $w=w'\cdot (3,5)$ and $1 \notin [2,6]$. The normal forms of $w$ and $w'$ are
\begin{align*}
\rr{w}&=NF(w) = s_5s_2s_3s_4s_1s_2s_3s_1, \\
\rr{w}'&=NF(w')= s_5s_1s_2s_3s_1s_2s_1.
\end{align*}
Moreover, $\rrt{w}{I}=\rrt{w}{4}= s_5s_2s_3\widehat{s_4}s_1s_2s_3s_1$ so that $\rr{w}_L=s_5s_2s_3$ and $\rr{w}_R=s_1s_2s_3s_1$. Note that $NF(w')\neq \rrt{w}{I}$. Since the first simple reflection ($s_5$) of $\rr{w}'$ and $\rrt{w}{I}$ coincide, and $\ell(\rr{w}_L)=3$, the algorithm will require two iterations. 
\begin{itemize}
\item Set $v=s_1s_2s_3s_1$ and $s=s_3$. Since $NF(sv)=s_1s_2s_3s_2s_1$, we obtain $j=3$ and $\text{degree}=(-1)^3=-1$. 
\item Set $v=s_1s_2s_3s_2s_1$ and $s=s_2$. Since $NF(sv)=s_1s_2s_3s_1s_2s_1$, we obtain $j=3$ and $\text{degree}=(-1)^3=-1$.
\end{itemize}
Hence, $\deg\left({\Phi}_{\rr{w}'}^{-1}\circ \Phi_{\rrt{w}{I}}\right)=(-1)\cdot(-1)=+1$. We may go further and compute the coefficient $c(w,w')$ according to Equation \eqref{eq:coefficient}: 
\begin{align*}
c(w,w^{\prime})&=(-1)^{I}\cdot\deg\left(\Phi_{\rr{w}'}^{-1}\circ \Phi_{\rrt{w}{I}} \right)\cdot(1+(-1)^{\height(\gamma^{\vee})}) \\
&=(-1)^4(+1)(1+(-1)^2)=+1.
\end{align*}
\end{ex}

\begin{rem}
    The algorithm also works if we proceed from right to left, that is, inserting elements from $\rr{w}_R$ into $\rr{w}_L$ until it reaches the normal form.
\end{rem}

\subsection{Computational Implementation}\label{subsec:computational}

The computation of the Poincaré polynomial for a Flag manifold $\mathbb{F}_\Theta = G/P_\Theta$ requires the explicit calculation of the coefficients $c(w,w')$, for each pair $w,w' \in \weyl$, as defined in Equation \eqref{eq:coefficient}.

Our implementation was developed in the open-source software SageMath~\cite{Sage}. This platform was chosen due to its Python-based environment and native support for symbolic computation. Specifically, we utilized its \texttt{WeylGroup} and \texttt{ChainComplex} libraries. The \textbf{\texttt{WeylGroup}} library provided essential functions to determine the covers of an element $w \in \weyl$, compute the height of coroots, and find the normal form of elements. Subsequently, the \textbf{\texttt{ChainComplex}} library allowed us to compute the Betti numbers from the boundary map. These numbers were then used to symbolically construct the final Poincaré polynomial.


The computational cost increases significantly with the dimension of the group. The most computationally intensive task is the construction of the boundary matrices; once these are determined, subsequent calculations, such as computing the Betti numbers, are relatively less consuming.

To address this bottleneck, we leveraged the fact that the primary computations for each element of the Weyl group are independent, making the problem inherently parallelizable. We implemented a multi-threaded algorithm to distribute the workload across multiple CPU cores. This parallel strategy, combined with other algorithmic optimizations and the computing resources of the National Center for High Performance Computing in São Paulo (CENAPAD-SP), enabled us to extend our calculations to the groups $B_7$, $C_7$, $D_7$, and $E_7$.

\begin{rem}
An optimization involves efficiently identifying the set of minimal coset representatives, $\weyl^\Theta = \{w \in \weyl \colon \ell(ws) > \ell(w) \text{ for all } s \in \Theta \}$. Our method relies on the property that every element $w \in \weyl$ is a minimal representative for a \textit{unique maximal} subset of simple roots, which is given by $\Delta_w = \Sigma \setminus D_R(w)$, where $D_R(w)$ is the set of right descents of $w$. This leads to a simple test: an element $w$ belongs to $\weyl^\Theta$ if and only if $\Theta \subseteq \Delta_w$. In practice, we precompute the set $\Delta_w$ for every $w \in \weyl$ and store it as a bit array. This allows the subset check to be performed with a fast bitwise operation.
\end{rem}

\section{Orientability of exceptional cases}\label{sec:orientability}

A direct application of the computation of Poincaré polynomials (see Section \ref{sec:formulas})  is the determination of the orientability of the real flag manifolds, at least for those of types $F_4,E_6$ and $E_7$. Once these polynomials are obtained, the question of orientability can be addressed in a straightforward manner by comparing the degree of the polynomial with the dimension of the corresponding flag manifold. This provides an effective and conceptually clear criterion, reducing a geometric property to an explicit algebraic verification. Moreover, agreement with previously known results confirms the validity of this approach and illustrates its effectiveness in capturing the topological structure of these spaces.

Table \ref{fig:table_orientability} presents the list of values for the real flag manifolds of split forms of $F_4$, $E_6$, and $E_7$. It is already known that the maximal flag manifolds are orientable (\cite{Koc95}, Cor. 1.1.10, \cite{PatraoSanMartinSantosSeco2012} Thm. 3.3). Regarding the partial flag manifolds, we have the following classification.

\begin{thm}
Let $\mathbb{F}_{\Theta}$ be a partial flag manifold of a split real form with $\Theta \subset \Sigma$. The orientable real flag manifolds are given by:
\begin{itemize}
\item If it is of type $F_4$, then $\Theta$ has connected type $A_2$ or $C_3$.
\item If it is of type $E_6$, then $\Theta$ has connected type $A_2$, $A_2\times A_2$, $A_4$, $D_4$ or $D_5$.
\item If it is of type $E_7$, then either $\Theta$ has connected type $A_2$, $A_1\times A_1\times A_1\times A_2$, $A_2\times A_2$, $A_4$, $D_4$, $D_5$, $A_2\times A_4$, $A_1\times A_2\times A_3$, $A_6$, $E_6$ or $\Theta$ is one of the subsets $\{\alpha_2,\alpha_5,\alpha_7\}$, $\{\alpha_2,\alpha_4,\alpha_5,\alpha_7\}$, $\{\alpha_2,\alpha_5,\alpha_6,\alpha_7\}$,$\{\alpha_2,\alpha_4,\alpha_5,\alpha_6,\alpha_7\}$.
\end{itemize}
\end{thm}

\begin{table}[ht]
    \centering
    \caption{Comparison between the degree of the Poincaré polynomial ($\beta^{\mathrm{top}}$) with the flag dimension ($\dim \mathbb{F}_{\Theta}$) in terms of the connected type of $\Theta$. In the case of $E_7$, the $*$ indicates that the corresponding $\Theta$ does not contain the set $\{\alpha_2,\alpha_5,\alpha_7\}$.
    }\label{fig:table_orientability}
    
    \begin{tabular}[t]{p{0.4\linewidth} p{0.55\linewidth}}
    \begin{minipage}[t]{\linewidth}
        \centering
        \textbf{\quad $F_4$} 
        \vspace{0.3em}
        
        \begin{tabular}{ccc}
            \toprule
            $\Theta$ & $\beta^{\mathrm{top}}$ & $\dim \mathbb{F}_{\Theta}$ \\
            \midrule
            $A_1$ & 21 & 23 \\
            $A_1\times A_1$ & 18 & 22 \\
            \textbf{$A_2$} & \textbf{21} & \textbf{21} \\
            $B_2$ & 18 & 20 \\
            $A_2\times A_1$ & 18 & 20 \\
            $B_3$ & 11 & 15 \\
            \textbf{$C_3$} & \textbf{15} & \textbf{15} \\
            \bottomrule
        \end{tabular}

        \vspace{1.5em}

        \textbf{\quad $E_6$} 
        \vspace{0.3em}
        
        \begin{tabular}{ccc}
            \toprule
            $\Theta$ & $\beta^{\mathrm{top}}$ & $\dim \mathbb{F}_{\Theta}$ \\
            \midrule
            $A_1$ & 33 & 35 \\
            $A_1\times A_1$ & 30 & 34 \\
            \textbf{$A_2$} & \textbf{33} & \textbf{33} \\
            $A_1\times A_1\times A_1$ & 27 & 33 \\
            $A_2\times A_1$ & 30 & 32 \\
            $A_2\times A_1 \times A_1$ & 27 & 31 \\
            $A_3$ & 26 & 30 \\
            \textbf{$A_2\times A_2$} & \textbf{30} & \textbf{30} \\
            $A_3\times A_1$ & 23 & 29 \\
            $A_2\times A_2\times A_1$ & 27 & 29 \\
            \textbf{$A_4$} & \textbf{26} & \textbf{26} \\
            $A_4\times A_1$ & 23 & 25 \\
            \textbf{$D_4$} & \textbf{24} & \textbf{24} \\
            $A_5$ & 15 & 21 \\
            \textbf{$D_5$} & \textbf{16} & \textbf{16} \\
            \bottomrule
        \end{tabular}
    \end{minipage}
    & 
    \begin{minipage}[t]{\linewidth}
        \centering
        \textbf{$E_7$} 
        \vspace{0.3em}
        
        \begin{tabular}{ccc}
            \toprule
            $\Theta$ & $\beta^{\mathrm{top}}$ & $\dim \mathbb{F}_{\Theta}$ \\
            \midrule
            $A_1$ & 60 & 62 \\
            $A_1 \times A_1$ & 59 & 61 \\
            ${A_1 \times A_1 \times A_1}^{\ast}$ & 54 & 60 \\
            $\{\alpha_2,\alpha_5,\alpha_7\}$ & \textbf{60} & \textbf{60} \\
            \textbf{$A_2$} & \textbf{60} & \textbf{60} \\
            $A_1 \times A_1 \times A_1 \times A_1$ & 57 & 59 \\
            $A_1 \times A_2$ & 58 & 59 \\
            $A_1 \times A_1 \times A_2$ & 54 & 58 \\
            \textbf{$A_1 \times A_1 \times A_1 \times A_2$} & \textbf{57} & \textbf{57} \\
            \textbf{$A_2 \times A_2$} & \textbf{57} & \textbf{57} \\
            $A_3$ & 53 & 57 \\
            $A_1 \times A_2 \times A_2$ & 54 & 56 \\
            ${A_1 \times A_3}^{\ast}$ & 50 & 56 \\
            $\{\alpha_2,\alpha_4,\alpha_5,\alpha_7\}$ & \textbf{56} & \textbf{56} \\
            $\{\alpha_2,\alpha_5,\alpha_6,\alpha_7\}$ & \textbf{56} & \textbf{56} \\
            $A_1 \times A_1 \times A_3$ & 53 & 55 \\
            $A_2 \times A_3$ & 50 & 54 \\
            \textbf{$A_1 \times A_2 \times A_3$} & \textbf{53} & \textbf{53} \\
            \textbf{$A_4$} & \textbf{53} & \textbf{53} \\
            $A_1 \times A_4$ & 50 & 52 \\
            \textbf{$D_4$} & \textbf{51} & \textbf{51} \\
            $A_1 \times D_4$ & 48 & 50 \\
            \textbf{$A_2 \times A_4$} & \textbf{50} & \textbf{50} \\
            ${A_5}^{\ast}$ & 42 & 48 \\
            $\{\alpha_2,\alpha_4,\alpha_5,\alpha_6,\alpha_7\}$ & \textbf{48} & \textbf{48} \\
            $A_1 \times A_5$ & 45 & 47 \\
            \textbf{$D_5$} & \textbf{43} & \textbf{43} \\
            $A_1 \times D_5$ & 40 & 42 \\
            \textbf{$A_6$} & \textbf{42} & \textbf{42} \\
            $D_6$ & 29 & 33 \\
            \textbf{$E_6$} & \textbf{27} & \textbf{27} \\
            \bottomrule
        \end{tabular}
        \end{minipage}
    \end{tabular}
\end{table}

The orientability of flag manifolds has already been studied by Kocherlakota \cite{Koc95} and Patrão \textit{et al.}\cite{PatraoSanMartinSantosSeco2012}. In the case of flag manifolds of split real forms, the following criterion is obtained (cf. \cite{PatraoSanMartinSantosSeco2012}, Eq. (5)): the flag $\mathbb{F}_{\Theta}$ is orientable if and only if for every root $\alpha \in \Sigma\setminus \Theta$,
\begin{align}\label{eq:orientability}
S(\alpha,\Theta)=\sum_{\beta \in \langle \Theta \rangle^+} \langle \alpha^\vee, \beta \rangle \equiv 0 \mod 2.
\end{align}

In \cite{PatraoSanMartinSantosSeco2012}, there is a section dedicated to the case of flags of split real forms where $S(\alpha,\Delta)$ is computed for the possible connected components $\Delta \subset \Theta$. If $\alpha$ is not connected to $\Delta$, then $S(\alpha,\Delta)=0$, and the only contribution occurs when $\alpha$ is connected to a root of $\Theta$. More specifically, if $\beta \in \langle \Theta \rangle^+$ is written as $\beta=c\delta + \gamma$ where $\delta$ is the only root of $\Delta$ connected to $\alpha$, then $\langle\alpha^\vee,\gamma\rangle =0$ and the only contribution is given by $\langle\alpha^\vee,\beta\rangle =c\langle\alpha^\vee,\delta\rangle$.

However, although the article intended to be exhaustive, one case that appears frequently in the context of the exceptional Lie algebras $E_6$, $E_7$, and $E_8$ was not addressed, namely, the cases where $\Delta$ is of type $A_n$, $3\leq n\leq 7$, and the root $\alpha \in \Sigma \setminus \Theta$ is connected to the root $\delta=\alpha_4$.

The general result is the following: for a diagram $\Delta=A_n$ and a root $\alpha \in \Sigma \setminus \Theta$ connected to the root $\delta$ in position $k$, $1\leq k\leq n$,
$$S(\alpha,\Delta)=-k(n-k+1).$$

In particular, if $k=1$ or $k=n$, $S(\alpha,\Delta)=-n$. 
Note also that, in terms of the mod $2$-congruence, $S(\alpha,\Delta)\not\equiv 0 \mod 2$ if and only if $k$ and $n$ are odd.

As an application, let us work out some examples. 

Consider the cases where $\Theta$ is of type $A_5$ and $A_6$, respectively, in $E_6$ and $E_7$. The corresponding diagrams are given by $\dynkin E{oxoooo}$ and $\dynkin E{oxooooo}$.
In both cases, the unique root $\alpha \in \Sigma \setminus \Theta$ is $\alpha=\alpha_2$, which is connected to $\delta=\alpha_4$. In $E_6$, we have $S(\alpha_2,\Delta=A_5)=-3(5-3+1)=-9\equiv 1 \mod 2$. In $E_7$, $S(\alpha_2,\Delta=A_6)=-3(6-3+1)=-12\equiv 0 \mod 2$. Hence, the corresponding flags are respectively non-orientable and orientable.

Now, if $\Theta$ is of type $A_4\times A_2$ in $E_7$, its diagram is given by $\dynkin E{ooooxoo}$. The unique root $\alpha \in \Sigma \setminus \Theta$ is $\alpha=\alpha_5$, which connects to the $A_4$ component through $\delta=\alpha_4$, for which $S(\alpha_5,\Delta=A_4)=-3(4-3+1)=-6\equiv 0 \mod 2$, and to the $A_2$ component through $\delta=\alpha_6$, for which $S(\alpha_5,\Delta=A_2)=-1(2-1+1)=-2\equiv 0 \mod 2$. Hence, this is an orientable flag.

\section{Final Remarks}\label{sec:final}

In a broader perspective, this work can be seen as the culmination of a series of papers devoted to the computation of the homology of real flag manifolds associated with normal forms. It also highlights the important role played by computational methods and the progress they have made possible: while earlier approaches relied on constructing permutation models for each individual type, we now have unified framework that applies to all types considered here. Nevertheless, several important problems remain open. Most notably, the computation of the homology groups of real flag manifolds of type $E_8$ is still beyond reach, mainly due to current computational limitations. Furthermore, the study of real flag manifolds associated with non split forms emerges as a natural next step, as very little is presently known about their topology. Another open direction concerns the Poincaré polynomials of types $B$, $C$, and $D$, for which the proposed formulas have so far been verified only computationally for certain values of $n$ and still lack a complete theoretical proof. Additionally, we were unable to obtain a closed formula for the cases in type $D$ where $\Theta$ contains both $a_{n-1}$ and $a_n$. Finally, we note that there are still very few references in the literature on the topology of exceptional flag manifolds, which underscores both the novelty of the present results and the need for further investigation in this area.

\bibliographystyle{amsplain} 
\bibliography{biblio}

\end{document}